\documentclass{amsart}
\usepackage{amsthm,amsmath,amssymb,enumitem,mathtools,mathrsfs,fullpage,tikz,graphicx,url}
\numberwithin{equation}{section}

\usepackage{comment}

\usepackage[normalem]{ulem}
\usepackage{xcolor}
\DeclareRobustCommand{\erase}{\bgroup\markoverwith{\textcolor{red}{\rule[.5ex]{2pt}{0.4pt}}}\ULon}

\theoremstyle{plain}
\newtheorem{thm}{Theorem}[section]
\newtheorem{cor}[thm]{Corollary}
\newtheorem{prop}[thm]{Proposition}

\newtheorem{lem}[thm]{Lemma}

\theoremstyle{definition}
\newtheorem{defn}[thm]{Definition}

\DeclareMathOperator{\conv}{conv}

\DeclareMathOperator{\init}{in}
\DeclareMathOperator{\rk}{rank}
\DeclareMathOperator{\md}{mod}

\newcommand{\FF}{\mathbb{F}}
\newcommand{\QQ}{\mathbb{Q}}
\newcommand{\RR}{\mathbb{R}}
\newcommand{\ZZ}{\mathbb{Z}}

\newcommand{\calB}{\mathcal{B}}
\newcommand{\calC}{\mathcal{C}}
\newcommand{\calG}{\mathcal{G}}
\newcommand{\grob}{\mathscr{G}}
\newcommand{\calI}{\mathcal{I}}
\newcommand{\calM}{\mathcal{M}}
\newcommand{\calS}{\mathcal{S}}
\newcommand{\calT}{\mathcal{T}}

\newcommand{\etilde}{\widetilde{e}}
\DeclareMathOperator{\SEP}{\Sigma}

\newcommand{\rob}[1]{\par \noindent
  \framebox{\begin{minipage}[c]{0.95 \textwidth}\color{blue} ROB'S NOTE:
      #1 \color{black}\end{minipage}}\par}      

\title{On the Ehrhart Theory of Generalized Symmetric Edge Polytopes}

\author[R. Davis]{Robert Davis}
\address{Department of Mathematics, Colgate University, Hamilton, NY, USA }
\email{rdavis@colgate.edu}

\author[A. Higashitani]{Akihiro Higashitani}
\address{Department of Pure and Applied Mathematics, Graduate School of Information
Science and Technology, Osaka University, Osaka, Japan}
\email{higashitani@ist.osaka-u.ac.jp}

\author[A. Higashitani]{Hidefumi Ohsugi}
\address{Department of Mathematical Sciences, School of Science, Kwansei Gakuin University, Sanda, Hyogo 669-1337, Japan}
\email{ohsugi@kwansei.ac.jp}

\date{\today}

\begin{document}

\begin{abstract}
    The symmetric edge polytope (SEP) of a (finite, undirected) graph is a centrally symmetric lattice polytope whose vertices are defined by the edges of the graph.
    SEPs have been studied extensively in the past twenty years.
    Recently, T\'othm\'er\'esz and, independently, D'Al\'i, Juhnke-Kubitzke, and Koch generalized the definition of an SEP to regular matroids, which are the matroids that can be represented by totally unimodular matrices.
    Generalized SEPs are known to have symmetric Ehrhart $h^*$-polynomials, and Ohsugi and Tsuchiya conjectured that (ordinary) SEPs have nonnegative $\gamma$-vectors.    

    In this article, we use combinatorial and Gr\"obner basis techniques to extend additional known properties of SEPs to generalized SEPs.
    Along the way, we show that generalized SEPs are not necessarily $\gamma$-nonnegative by providing explicit examples.
    We prove that the polytopes we construct are ``nearly'' $\gamma$-nonnegative in the sense that, by deleting two particular elements from the matroid, one obtains SEPs for graphs that are $\gamma$-nonnegative.
    This provides further evidence that Ohsugi and Tsuchiya's conjecture holds in the ordinary case.
\end{abstract}

\maketitle

\section{Introduction}
    A \emph{lattice polytope} $P$ in $\RR^n$ is the convex hull of finitely many points in $\ZZ^n$.
    That is, $P$ is a lattice polytope if and only if there are some $v_1,\dots,v_k \in \ZZ^n$ such that
    \[
        P = \conv\{v_1,\dots,v_k\} = \left\{
        \left.\sum_{i=1}^k \lambda_iv_i \ \right| \ \lambda_1,\dots,\lambda_k \geq 0, \, \sum_{j=1}^k \lambda_j = 1\right\}.
    \]
    The main objects of study in this paper will be lattice polytopes arising from (finite, undirected) graphs and matroids.    
    First, given an undirected graph $G = ([n],E)$ where $[n] = \{1,\dots,n\}$, the \emph{symmetric edge polytope} (or \emph{SEP}) associated to $G$ is
    \[
        \Sigma(G) = \conv\{\pm(e_i - e_j) \in \RR^n \mid ij \in E\}.
    \]
    This polytope was introduced in \cite{MHNOH} and has been the object of much study for its algebraic and combinatorial properties \cite{DDM, HJM, OhsugiHibi, OhsugiTsuchiya21}, as well as its applications to problems arising from engineering \cite{Chen, ChenDavisKorchevskaia, ChenDavisMehta}.

    SEPs have recently been generalized to regular matroids in two ways.
    Suppose $A$ is a totally unimodular matrix representing a matroid $\calM$.
    T\'othm\'er\'esz \cite{Tothmeresz}, working within the context of oriented matroids, defined the \emph{root polytope} of $\calM$ to be the convex hull of the columns of $A$.
    Independently, D'Al\`i, Juhnke-Kubitzke, and Koch \cite{DJKK} defined the \emph{generalized symmetric edge polytope} (or \emph{generalized SEP}) associated to $\calM$ as 
    \[
        \Sigma(\calM)= \conv\{\pm u \mid u \text{ is a column of }A\}.
    \]
    Thus, $\Sigma(\calM)$ may be considered the root polytope of the matroid represented by the matrix $[A \mid -A]$, which is also totally unimodular.
    Because we will not be working with oriented matroids, and because we are focused on the polytopes $\Sigma(\calM)$ in this work, we will follow the terminology and notation in \cite{DJKK}.
    
    Although the definition of $\Sigma(\calM)$ depends on the choice of $A$, the polytope is well-defined up to unimodular equivalence (\cite[Proposition 3.2]{Tothmeresz} and \cite[Theorem~3.2]{DJKK}). 
    Unimodularly equivalent lattice polytopes have the same Ehrhart polynomial and consequently, the same Ehrhart series.
    It is common, then, for Ehrhart theorists to identify unimodularly equivalent lattice polytopes.
    In \cite{DJKK}, the authors proved that many properties held by ordinary SEPs also hold for generalized SEPs. 
    Further, this more general setting allowed the authors to show that two SEPs are unimodularly equivalent if and only if the underlying matroids are isomorphic \cite[Theorem 4.6]{DJKK}.
    In addition to these aspects of SEPs, much attention has been given to their Ehrhart theory, which we briefly introduce here.
    
    The function $m \mapsto |mP \cap \ZZ^n|$, defined on the positive integers, agrees with a polynomial $L_P(m)$ of degree $\dim(P)$, called the \emph{Ehrhart polynomial} of $P$.
    Consequently, the \emph{Ehrhart series} of $P$, which is the (ordinary) generating function for the sequence $\{L_P(m)\}_{m \geq 0}$, may be written as the rational function
    \[
        E(P;t) = \sum_{m \geq 0} L_P(m)t^m = \frac{h_0^* + h_1^*t + \cdots + h_d^*t^d}{(1-t)^{\dim(P)+1}}
    \]
    for some $d \leq \dim(P)$ with $h_d^* \neq 0$.
    We denote the numerator by $h^*(P;t)$ and refer to it as the \emph{$h^*$-polynomial} of $P$.
    The \emph{$h^*$-vector} of $P$, $h^*(P)$, is the list of coefficients $h^*(P) = (h^*_0,\dots,h^*_d)$ of $h^*(P;t)$.

    The $h^*$-vectors of polytopes have been the focus of much study in recent decades, as they lie within the intersection of polyhedral geometry, commutative algebra, algebraic geometry, combinatorics, and much more (see, e.g., \cite{StanleyGreenBook, sturmfels}).
    For example, when the $h^*$-vector is \emph{symmetric}, meaning $h^*_i = h^*_{d-i}$ for each $i = 0,\dots,d$, then the associated semigroup algebra
    \[
        K[P] = K[x^vz^m \mid v \in mP \cap \ZZ^n] \subseteq K[x_1^{\pm},\dots,x_n^{\pm},z],
    \]
    over a field $K$, where $x^v = x_1^{v_1}\cdots x_n^{v_n}$, is Gorenstein.
    In fact, the converse holds as well.
    Determining when $K[P]$ is Gorenstein can be done entirely using polyhedral geometry, so, from this perspective, there is a complete characterization of when $h^*(P)$ is symmetric.

    On the other hand, a more mysterious property held by some $h^*$-vectors is unimodality.
    We call $h^*(P)$ \emph{unimodal} if there is some index $0 \leq r \leq d$ for which $h^*_0 \leq \cdots \leq h^*_r$ and $h^*_r \geq \cdots \geq h^*_d$.
    Unlike symmetry, there is no complete characterization of when an $h^*$-vector is unimodal.
    There are various sets of sufficient conditions that ensure a polytope has a unimodal $h^*$-vector and various sets of necessary conditions \cite{BraunUnimodality}, but no set of conditions capturing both simultaneously.

    When an $h^*$-vector is already known to be symmetric, new avenues for proving $h^*$-unimodality appear.
    For instance, in this setting, the $h^*$-polynomial may be expressed as
    \begin{equation}\label{eq: hstar gamma equivalence}
        \sum_{i=0}^d h^*_it^i = \sum_{i=0}^{\lfloor d/2 \rfloor} \gamma_it^i(1+t)^{d-2i} = (1+t)^d \sum_{i=0}^{\lfloor d/2 \rfloor} \gamma_i \left( \frac{t}{(1+t)^2}\right)^i
    \end{equation}
    for certain choices of $\gamma_i \in \QQ$.
    This leads to the study of the \emph{$\gamma$-polynomial} of $P$,
    \[
        \gamma(P;t) = \sum_{i=0}^{\lfloor d/2 \rfloor} \gamma_it^i,
    \]
    and of the \emph{$\gamma$-vector} of $P$, $\gamma(P) = (\gamma_0,\dots,\gamma_{\lfloor d/2 \rfloor})$.
    Note that if $\gamma_i \geq 0$ for each $i$, then $h^*_i \geq 0$ for each $i$ as well.
    This is not necessary, though: the $h^*$-vector of $(1,1,1)$, realized by the lattice triangle with vertices $(1,0)$, $(0,1)$, and $(-1,-1)$, is unimodal but has a $\gamma$-vector of $(1,-1)$.
    Nevertheless, $\gamma$-vectors remain the topic of intense study, due in part to wide-open conjectures surrounding their nonnegativity.
    The most famous of these is \emph{Gal's Conjecture}: that the $h$-vector of a flag homology sphere (a pure $n$-dimensional simplicial complex having the same homology groups as the $n$-sphere and
    such that all minimal non-faces have dimension $1$) is $\gamma$-nonnegative. 
    To an Ehrhart theorist, the conjecture states that the $h^*$-polynomial of a lattice polytope with a  flag (all minimal non-faces have dimension $1$), regular, unimodular triangulation is $\gamma$-nonnegative.
    
    In this article, we use combinatorial and Gr\"obner basis techniques to extend additional known properties of SEPs to generalized SEPs.
    Along the way, we show that generalized SEPs are not always $\gamma$-nonnegative by providing an infinite class of explicit examples (Theorem~\ref{thm: not nonnegative}).
    These examples are ``nearly'' $\gamma$-nonnegative in the sense that there are two particular elements of the matroid which, when deleted, result in SEPs for graphs that are $\gamma$-nonnegative.
    We present and prove a formula for their $h^*$-polynomials (Theorem~\ref{thm: h star of gamma}), deduce the normalized volumes of these SEPs (Corollary~\ref{cor: normalized volume}), and compute their $\gamma$-polynomials (Theorem~\ref{thm: gamma vec}).
    This provides further evidence that Ohsugi and Tsuchiya's conjecture holds in the ordinary case.

    A brief structure of the article is as follows.
    In Section~\ref{sec: background} we discuss the necessary background for matroids and the relationships among Ehrhart theory, triangulations of polytopes, and toric ideals.
    Section~\ref{sec: extensions} discusses how various important results for ordinary SEPs extend to generalized SEPs.
    In Section~\ref{sec: counterexample} we present explicit examples of generalized SEPs whose $h^*$-polynomials are not $\gamma$-nonnegative (Theorem~\ref{thm: not nonnegative}) and examine the consequences of making minor variations to them.
    We also state two of the main results of this article: Theorem~\ref{thm: h star of gamma} and Theorem~\ref{thm: gamma vec}.
    In Section~\ref{sec: proof} we both prepare for and present the full proof of Theorem~\ref{thm: h star of gamma}, which follows a combinatorial argument.
    Finally, in Section~\ref{sec: proof2}, we use a generating function argument to prove Theorem~\ref{thm: gamma vec}.

\section{Background}\label{sec: background}

To keep this article as self-contained as possible, this section provides the necessary background on matroids, Ehrhart theory, and generalized SEPs.
For a comprehensive reference on the first two topics, we refer the reader to \cite{Oxley} for matroids and \cite{BeckRobins} for Ehrhart theory. 
Throughout this article we use the convention of writing $A \cup x$ and $A - x$ for the union or deletion, respectively, of $A$ by a one-element set $\{x\}$.

\subsection{Matroids}

A \emph{matroid} $\calM$ is a pair $(E,\calI)$ where $E$ is finite and $\calI \subseteq 2^E$ satisfies the following three properties:
\begin{enumerate}[label=(\roman*)]
    \item $\emptyset \in \calI$;
    \item if $I \in \calI$ and $J \subseteq I$, then $J \in \calI$; and
    \item if $I, J \in \calI$ and $|I| < |J|$, then there is some $x \in J - I$ such that $I \cup x \in \calI$.
\end{enumerate}
The set $E$ is called the \emph{ground set} of $\calM$ and the elements of $\calI$ are called the \emph{independent sets} of $\calM$.

Matroids are useful for unifying various notions of ``independence'' in mathematics.
For example, if $G$ is a (finite) graph, then there is a corresponding matroid $M(G)$ whose ground set is $E(G)$, the edge set of $G$, and whose independent sets are the sets of edges that form acyclic subgraphs of $G$.
This matroid is called the \emph{cycle matroid} of $G$, and any matroid that is isomorphic to $M(G)$ for some $G$ is a \emph{graphic matroid}.

Another fundamental example of a matroid uses the columns of a $k \times n$ matrix $A$ as the ground set.
Here, the independent sets are the columns of $A$ that form linearly independent sets.
This matroid, denoted $M(A)$, is the \emph{vector matroid} of $A$, and any matroid isomorphic to $M(A)$ for some matrix $A$ is called a \emph{representable matroid}.
More precisely, for a field $K$, we say that a matroid is \emph{$K$-representable} if it is isomorphic to a vector matroid whose ground set consists of vectors in $K^k$.
Hence, to say a matroid is representable is to say it is $K$-representable for some field $K$.
On the other hand, a matroid is \emph{regular} if it is $K$-representable over every field $K$.
Importantly, graphic matroids are regular, cographic matroids (that is, matroids whose duals are graphic) are regular, and a matroid is regular if and only if its dual is regular. 

The \emph{rank} of a matroid $\calM = (E,\calI)$ is defined as 
\[
    \rk(\calM) = \max_{I \in \calI} |I|.
\]
When $\calM$ is graphic, $\rk(\calM)$ is the size of a maximal spanning forest for any $G$ satisfying $\calM \cong M(G)$.
When $\calM$ is representable, $\rk(\calM) = \rk(A)$ for any matrix $A$ satisfying $\calM \cong M(A)$.
With these definitions in hand, we can now describe some of the many benefits of working with regular matroids.

\begin{thm}[{\cite[Definition/Theorem 2.3]{DJKK}}]\label{thm:defthm}
    A positive-rank matroid $\calM$ is regular if and only if any of the following hold:
    \begin{enumerate}[label={\rm (\roman*)}]
        \item $\calM$ is representable over every field.
        \item $\calM \cong M(A)$ for some matrix $A$ in which every minor is $\pm 1$ or $0$.
        \item $\calM \cong M(A)$ for some full-rank matrix $A$ in which every maximal minor is $\pm 1$ or $0$.
    \end{enumerate}
\end{thm}

Note that a matrix satisfying the minor condition in (ii) is called \emph{totally unimodular} and a matrix satisfying the minor condition in (iii) is called \emph{weakly unimodular}.
Borrowing more terminology from linear algebra, an element of $\calI$ of cardinality $\rk(\calM)$ is called a \emph{basis} of $\calM$.
The set $\calB = \calB(\calM)$ of bases, together with the ground set $E$, can be used to define a matroid without explicit mention of independent sets; when this is done, we usually write $\calM = (E, \calB)$.

Now, from a matroid $\calM = (E,\calB)$ we may produce its \emph{dual} matroid $\calM^* = (E, \calB^*)$ by setting
\[
    \calB^* = \{E - B \subseteq E \mid B \in \calB\}.
\]
Many properties of $\calM^*$ can be easily deduced from $\calM$. 
For example, $\calM^*$ is representable if and only if $\calM$ is representable.
This does not hold for graphic matroids, though: if $\calM$ is graphic, then $\calM^*$ is graphic if and only if $\calM \cong M(G)$ for some planar graph $G$.

There are additional notions in graph theory that have corresponding matroidal analogues.
For example, a \emph{circuit} in a matroid is a minimal dependent set; this is the analogue of a cycle in a graph.
Letting $\calC = \calC(\calM)$ denote the set of circuits of a matroid, we define the matroid to be \emph{bipartite} if each of the circuits has even size. 
This further allows one to call a $\FF_2$-representable, or \emph{binary}, matroid \emph{Eulerian} if it is the dual of a bipartite matroid \cite[Proposition 9.4.1]{Oxley}.
These definitions reflect the fact that a connected, planar graph is Eulerian if and only if its dual is bipartite.

A \emph{loop} of $\calM$ is a single element of $E$ that forms a circuit; equivalently, a loop is an element that is in no basis of $\calM$. 
On the other hand, a \emph{coloop} of $\calM$, called by some authors an \emph{isthmus}, is a single element of $E$ that is in every basis of $\calM$.
Now, given a matroid $\calM = (E, \calI)$ and an element $e \in E$, the \emph{deletion} of $e$ from $\calM$ is the matroid $\calM - e$ having ground set $E - e$ and independent sets $\{ I \in \calI \mid e \notin I\}$. 
The \emph{contraction} of $\calM$ by $e \in E$, where $e$ is not a loop, is denoted $\calM / e$ and has ground set $E - e$ and independent sets $\{I - e \mid e \in I \in \calI\}$.
When $e$ is a loop, we define $\calM / e$ as $\calM - e$; this is because a loop is in no independent set, causing $\{I - e \mid e \in I \in \calI\}$ to be empty, violating the property $\emptyset \in \calI$.

Deletion and contraction are dual operations: $(\calM - e)^* = \calM^*/e$.
Many binary operations on graphs extend to matroids as well:
the \emph{direct sum}, or \emph{1-sum}, of matroids $\calM_1$ and $\calM_2$ having disjoint ground sets $E_1$ and $E_2$ is $\calM_1 \oplus \calM_2 = (E_1 \cup E_2, \calI_3)$, where $\calI_3$ consists of independent sets of the form $I = I_1 \cup I_2$ where $I_1 \in \calI_1$ and $I_2 \in \calI_2$.

To define the next operation we assume that $E_1 \cap E_2 = \{p\}$ and call this element the \emph{basepoint}.
The \emph{parallel connection}, denoted $P(\calM_1,\calM_2)$, has ground set $E_1 \cup E_2$ and is most efficiently defined through its circuits: if $p$ is neither a loop nor a coloop of $\calM_1$, then
\[
    \calC(P(\calM_1,\calM_2)) = \, \calC(\calM_1) \cup \calC(\calM_2) \cup \{(C_1 \cup C_2) - p \mid p \in C_i \in \calC(\calM_i) \text{ for each } i\}.
\]
If $p$ is a loop of $\calM_1$, then we set
\[
P(\calM_1, \calM_2) = \calM_1 \oplus (\calM_2 / p). 
\]
Similarly, if $p$ is a coloop of $\calM_1$, then we set
\[
P(\calM_1, \calM_2) = (\calM_1-p) \oplus \calM_2. 
\]
Of course, if $\calM_1$ and $\calM_2$ have disjoint ground sets, then a point of each ground set may be selected and relabeled $p$ to fall into this definition, although the resulting parallel connection may then depend on which points were selected.
An important property of this operation is that it preserves regularity: if $\calM_1$ and $\calM_2$ are both regular matroids, then so is $P(\calM_1,\calM_2)$ \cite[Corollary 7.1.25]{Oxley}.


\subsection{Triangulating Tree Sets and Toric Algebra}

It is frequently more tractable to gather the desired information about a polytope by decomposing it into subpolytopes.
A \emph{dissection}, or \emph{subdivision}, of a polytope $P$ is a collection $\calS$ of full-dimensional subpolytopes of $P$ whose union is $P$ and such that, for any $S_1, S_2 \in \calS$, the intersection $S_1 \cap S_2$ is a face of both. 
A \emph{triangulation} of $P$ is a dissection in which every $S \in \calS$ is a simplex.
There are many ways to create triangulations of a polytope; we will be concerned with one that has been recently introduced in the study of ordinary SEPs in particular.

Given an oriented subgraph $H$ of a graph $G$ on $[n]$, let $Q_H = \conv\{e_i - e_j \mid ij \mbox{ is a directed edge in } H\}$.
Thus, if $T$ is a spanning tree of $G$, then $Q_T$ is a simplex contained in $\Sigma(G)$.
In fact, $Q_T$ is a $(\dim(\Sigma(G))-1)$-dimensional simplex contained in the boundary of $\Sigma(G)$.
Further still, $Q_T$ is a \emph{unimodular} simplex: its vertices form an affine lattice basis of the affine span of $Q_T$.
Equivalently, a unimodular simplex has normalized volume $1$.
It follows that the simplex $\conv\{Q_T \cup \{0\}\}$ is a full-dimensional, unimodular simplex contained in $\Sigma(G)$.
It is desirable, then, to find a collection $\calT$ of oriented spanning trees of $G$ so that the set $\{Q_T \mid T \in \calT\}$ forms a triangulation of the boundary of $\SEP(G)$.
Such a set $\calT$ is called a \emph{triangulating tree set} for $G$.

Of course, if $\calT$ is a triangulating tree set for $G$, then the normalized volume of $\Sigma(G)$ is exactly $|\calT|$.
However, there is more information about $\Sigma(G)$ to be extracted from $\calT$.
To describe this information, consider a fixed vertex $v$ in an oriented tree $T$. 
For each edge $e \in T$ we say that \emph{$e$ points away from $v$ in $T$} if $v$ is contained in the same component as the tail of $e$ in $T - e$.
Otherwise, we say \emph{$e$ points towards $v$ in $T$}.

\begin{thm}[{\cite[Theorem 1.2]{KalmanTothmeresz}}]\label{thm: h-star contributions}
    Let $G$ be a connected graph, let $\calT$ be a triangulating tree set for $G$, and let $v$ be any vertex of $G$.
    Then the $i^{th}$ entry in $h^*(\Sigma(G))$ is
    \[
        |\{T \in \calT \mid T \text{ has exactly } i \text{ edges pointing away from } v\}|.
    \]
\end{thm}

In order to detect when we have a triangulating tree set, we turn to Gr\"obner basis techniques.
Here, we will introduce the necessary algebraic background to prove our main results. 
We largely follow the framework given in \cite{HJM}.
For more details about monomial orders and their relationships to triangulations of polytopes, see \cite[Chapter 8]{sturmfels}.

Recall that, given a lattice polytope $P \subseteq \RR^n$ and a field $K$, we may construct a semigroup algebra 
\[
    K[P] = K[x^va^m \mid v \in mP \cap \ZZ^n,\, m \in \ZZ_{>0}] \subseteq K[x_1^{\pm 1}, \dots, x_n^{\pm 1},a].
\]
The \emph{toric ideal} of $P$ is defined as the kernel of the map
\[
    \pi_P: K[t_v \mid v \in P \cap \ZZ^n] \to K[P]
\]
defined by setting $\pi_P(t_v) = x^va$; this ideal is denoted $I_P$.

Before continuing, recall that a \emph{monomial order} on a polynomial ring $R = K[t_1,\dots,t_n]$ is a relation $\prec$ on the monomials $t^a$ of $R$ such that
\begin{itemize}
    \item $\prec$ is a total ordering on the monomials of $R$,
    \item if $t^a \prec t^b$, then $t^{a+c} \prec t^{b+c}$ for any $c \in \ZZ_{\geq 0}^n$, and
    \item $1 \prec t^a$ for every $a \in \ZZ_{\geq 0}^n$.
\end{itemize}
There are many monomial orders one may care about; the one that will be of most interest to us is the \emph{graded reverse lexicographic (grevlex)} order, defined as follows.

Begin by setting an order $t_{i_1} < t_{i_2} < \cdots < t_{i_n}$ on the variables of $R$.
For $a = (a_1,\dots,a_n) \in \ZZ_{\geq 0}^n$, set $\left|a\right| = \sum a_i$.
Grevlex declares $t^a \prec t^b$, where $a = (a_1,\dots,a_n)$ and $b = (b_1,\dots,b_n)$, if $\left|a\right| < \left|b\right|$ or if both $\left|a\right| = \left|b\right|$ and $b_{i_k} - a_{i_k} < 0$ where $k$ is the largest value for which $b_{i_k} - a_{i_k}$ is nonzero.

Returning to the setting of toric ideals, consider a monomial order $\prec$ on $K[t_v \mid v \in P \cap \ZZ^n]$ and any ideal $I$ of this algebra.
For each polynomial $f \in I$, its \emph{initial} term with respect to $\prec$, denoted $\init_{\prec}(f)$, is the term of $f$ that is greatest with respect to $\prec$.
The \emph{initial ideal} of $I$ with respect to $\prec$ is 
\[
	\init_{\prec}(I) = ( \init_{\prec}(f) \mid f \in I).
\]
A \emph{Gr\"obner basis} of $I$, which we denote by $\grob$, is a finite generating set of $I$ such that $\init_{\prec}(I) = (\init_{\prec}(g)\ |\ g \in \grob)$.
We call $\grob$ \emph{reduced} if each element has a leading coefficient of $1$ and if for any $g_1, g_2 \in \grob$, $\init_{\prec}(g_1)$ does not divide any term of $g_2$.
There can be many Gr\"obner bases of $I$ with respect to $\prec$, but there is exactly one reduced Gr\"obner basis of $I$ with respect to $\prec$.
We stress ``with respect to $\prec$'' since two different monomial orders may result in two distinct reduced Gr\"obner bases of $I$; uniqueness comes once a choice of $\prec$ has been made.

Now suppose $P \subseteq \RR^n$ is an $n$-dimensional lattice polytope and $P \cap \ZZ^n = \{l_1,\ldots,l_s\}$.
Choose a weight vector $w = (w_1,\ldots,w_s) \in \RR^s$ such that the polytope
\[ 
    P_w = \conv\{(l_1,w_1),\ldots,(l_s,w_s)\} \subseteq \RR^{n+1}
\]
is $(n+1)$-dimensional, i.e., $P_w$ does not lie in an affine hyperplane of $\RR^{n+1}$.
Some facets of $P_w$ will have outward-pointing normal vectors with a negative last coordinate.
By projecting these facets back to $\RR^n$, we obtain the facets of a polyhedral subdivision $\Delta_w(P)$ of $P$.
Any triangulation that can be obtained in this way for an appropriate choice of $w$ is called a \emph{regular} triangulation.

Importantly, regular unimodular triangulations are in one-to-one correspondence with reduced Gr\"obner bases whose initial terms are all squarefree \cite[Corollary 8.9]{sturmfels}.
The correspondence is as follows: if $S = \{l_{i_1},\dots,l_{i_r}\}$ is a set of lattice points in $P$, then $\conv(S)$ is a simplex of $\Delta_w(P)$ if and only if no initial term in the reduced Gr\"obner basis divides $\prod_{l\in S} t_l$. 
This comes from the fact that when the reduced Gr\"obner basis consists of binomials with squarefree initial terms, then $\init_{\prec}(I)$ is the Stanley-Reisner ideal of $\Delta_w(P)$.
For further details of this correspondence, see, e.g., \cite[Chapter 8]{sturmfels}.
Additionally, it is well-known that any monomial order -- hence, those obtained as a grevlex order -- can be represented by a weight vector \cite[Proposition 1.11]{sturmfels}.


In the case of $P = \Sigma(\calM)$ for a regular matroid $\calM = (E, \calI)$, we can identify $K[t_v \mid v \in P \cap \ZZ^n]$ with 
\begin{equation}\label{eq: semigp}
    K\left[\{z\} \cup \bigcup_{e \in E} \{x_e,y_e\} \right].
\end{equation}
Here, we let $x_e$ correspond to the column of a fixed matrix $M$ realizing $\calM$ that corresponds to $e$, and we will let $y_e$ correspond to the negative of this column.
The variable $z$, then, corresponds to the origin. 
So, for each $e \in E$, we can treat $x_e$ and $y_e$ as encoding an orientation of $e$.
For ease of notation, if we are considering $e \in E$ with a particular orientation, then we let $p_e$ denote the corresponding variable and $q_e$ the variable corresponding to the opposite orientation.
From \cite[Theorem 6.11]{DJKK}, we have the following immediately.

\begin{lem}\label{lem: initial terms}
    Let $\calM$ be a regular matroid on the ground set $E$, let $z < x_{e_1} < y_{e_1} < \cdots < x_{e_n} < y_{e_n}$ be an order on the lattice points of $\Sigma(\calM)$, and let $\prec$ be the grevlex order on \eqref{eq: semigp} induced from this ordering on the variables.
    The collection of all binomials of the following types forms a Gr\"obner basis of $I_{\Sigma(\calM)}$ with respect to the grevlex order:
    \begin{enumerate}[label={\rm (\roman*)}]
        \item For every $2k$-element circuit $C$ of $\calM$ and choice of orientation for each $e \in C$, and any $k$-element subset $I \subseteq C$ not containing the weakest element among those of $C$,
        \[
            \prod_{e \in I} p_e - \prod_{e \in C - I} q_e.
        \]
        \item For every $(2k+1)$-element circuit $C$ of $\calM$ and choice of orientation for each $e \in C$, and any $(k+1)$-element subset $I \subseteq C$,
        \[
            \prod_{e \in I} p_e - z\prod_{e \in C - I} q_e.
        \]
        \item For any $e \in E$, $x_ey_e - z^2$.
    \end{enumerate}
    In each case, the binomial is written so that the initial term has positive sign.
\end{lem}

\section{Extensions of known properties of SEPs}\label{sec: extensions}

Before providing the extensions of known properties of ordinary SEPs to generalized SEPs, 
we recall the formula of their dimensions. 
\begin{prop}[{cf. \cite[Theorem 4.3]{DJKK}}]\label{prop:dim}
    Let $\calM$ be a regular matroid on the ground set $E$. Then $\dim \Sigma(\calM)=\rk(\calM)$. 
    In particular, given a connected graph $G$ with $n$ vertices, we have $\dim \Sigma(G)=n-1$. 
    Moreover, we have $\dim \Sigma(\calM^*)=|E|-\rk (\calM)$. 
\end{prop}

Let $P \subseteq \RR^n$ be an $n$-dimensional lattice polytope containing the origin in its interior.
The \textit{(polar) dual polytope} of $P$ is
$P^\vee =\{y \in \RR^n  \mid  \langle x, y \rangle \leq 1 \ \text{for all}\  x \in P \},
$
where $\langle x, y \rangle$ is the usual inner product of $\RR^n$.
Then $P$ is called \textit{reflexive} if $P^\vee$
is also a lattice polytope.
It is known that $P$ is reflexive if and only if 
$h^*(P; t)$ is symmetric (see \cite{HibiReflexive, StanleyGorenstein}).
In particular, $\SEP(\calM)$ is reflexive for any regular matroid $\calM$ of positive rank \cite{DJKK}.

Let $P \subseteq \RR^n$ and $Q \subseteq \RR^m$ be lattice polytopes.
Then the \textit{free sum} (or \textit{direct sum}) $P \oplus Q \subseteq \RR^{m+n}$ is defined as
the convex hull of the set $(P \times 0_m) \cup (0_n \times Q)$ where $0_m \in \RR^m$ and $0_n \in \RR^n$
are the origins.

\begin{prop}[{\cite[Theorem 1]{Braunfreesum}}]
    Let $P$ be a reflexive polytope and let $Q$ be 
    a lattice polytope containing the origin in its interior. Then we have 
    \[
        h^*(P \oplus Q;t) =h^*(P; t) h^*(Q; t).
    \]
\end{prop}

Since $\SEP(\calM_1 \oplus \calM_2) = \SEP(\calM_1)\oplus \SEP(\calM_2)$ for regular matroids $\calM_1$ and $\calM_2$, 
we have the following.

\begin{cor}\label{free_sum}
Let $\calM_1 \oplus \calM_2$ be the direct sum of regular matroids $\calM_1$ and $\calM_2$.
Then we have 
\[h^*(\SEP(\calM_1 \oplus \calM_2);t) =h^*(\SEP(\calM_1); t) h^*(\SEP(\calM_2); t).\]
\end{cor}

First, we extend \cite[Proposition 5.4]{OhsugiTsuchiya21}, using the same overall strategy used in its proof.

\begin{prop}[{cf. \cite[Proposition 5.4]{OhsugiTsuchiya21}}]\label{prop: contraction}
    Let $\calM$ be a bipartite regular matroid on $E$.
    For any $e \in E$,
    \[
        h^*(\SEP(\calM);t) = (1+t)h^*(\SEP(\calM / e); t).
    \]
\end{prop}

\begin{proof}
First note that since $\calM$ is bipartite, $e$ is not a loop.
     Throughout this proof, we use the grevlex  order.
     Let $A$ be a full-rank, totally unimodular matrix of rank $r$ representing $\calM$, and denote its columns by $a_1,\dots,a_n$ so that $e = a_1$.
     Also, let $A_e$ be a matrix of rank $r-1$ representing $\calM/e$.
     Since $\calM$ is bipartite, the Gr\"obner basis $\grob$ of $I_{\SEP(\calM)}$ from Lemma~\ref{lem: initial terms} consists of the binomials of the form (i) and (iii).
     Let $A'$ be the rank-{$r$} matrix
     \[
    A' = \left[\begin{array}{@{}c|c@{}}
            \begin{matrix} 0 \\
            \vdots \\ 0 \end{matrix} & A_e\\
            \hline
            1 & \begin{matrix} 0 & \cdots & 0 \end{matrix} 
        \end{array}\right],
    \]
   and denote its columns by $a_1',\dots,a_n'$. Let $\calM'$ be the matroid represented by $A'$.
Then we have $\calM' = a_1'  \oplus \calM /e$ and 
the set of circuits in $\calM'$ coincides with that in $\calM/e$. 
From \cite[Proposition~3.1.10]{Oxley},
\begin{itemize}
    \item[(a)]
    the set of circuits in $\calM'$ consists of minimal nonempty members of 
$\{ C - e \mid C \in \calC(\calM)\}$.
\end{itemize}
Moreover, from \cite[Exercise 3.1.2]{Oxley}, given a circuit $C$ in $\calM$,
\begin{itemize}
    \item[(b)]
    if $e \notin C$, then $C$ is a union of circuits in $\calM'$; and
    \item[(c)]
    if $e \in C$, then $C - e$ is a circuit in $\calM'$. 
\end{itemize}
Let $\grob'$ be the set    
of all binomials $b$ satisfying one of the following:
     \[
         b = \prod_{a_i \in I} p_{a_i'} - \prod_{a_i \in C - I} q_{a_i'},
     \]
     where $C$ is a set of columns of $A$ corresponding to a circuit in $\calM$ of size $2k$ and $a_1 \notin C$, and $I$ is a $k$-subset of $C$ such that $a_j \notin I$ where $j = \min\{i \mid a_i \in C\}$; or
     \[
         b = \prod_{a_i \in I} p_{a_i'} - z\prod_{a_i \in C - I} q_{a_i'},
     \]
     where $C \cup a_1$ corresponds to a circuit in $\calM$ of size $2k+2$ and $I$ is a $(k+1)$-subset of $C$; or
     \[
         b = x_{a_i'}y_{a_i'} - z^2
     \]
     for $1 \leq i \leq n$.
     From (a), it is easy to see that the Gr\"obner basis of $I_{\Sigma(\calM')}$ from Lemma~\ref{lem: initial terms} is a subset of $\grob'$.
     If a circuit $C$ in $\calM$ is the union of circuits $C_1, \dots, C_r$ of $\calM'$, then 
     \[\begin{aligned}
      \pi_{\Sigma(\calM')} \left( \prod_{a_i \in I} p_{a_i'} - \prod_{a_i \in C - I} q_{a_i'} \right) &=
     \prod_{j=1}^r \pi_{\Sigma(\calM')} \left( \prod_{a_i \in I \cap C_j} p_{a_i'} \right) - 
    \prod_{j=1}^r \pi_{\Sigma(\calM')} \left( \prod_{a_i \in (C- I) \cap C_j} q_{a_i'} \right) \\
    &=
     \prod_{j=1}^r \pi_{\Sigma(\calM')} \left( \prod_{a_i \in I \cap C_j} p_{a_i'} - \prod_{a_i \in (C- I) \cap C_j} q_{a_i'} \right) \\
    &= 0.
     \end{aligned}\]
     Hence $ \prod_{a_i \in I} p_{a_i'} - \prod_{a_i \in C - I} q_{a_i'}$
     belongs to $I_{\Sigma(\calM')}$. From 
     (b) and (c), it follows that     
     $\grob'$ is a subset of $I_{\Sigma(\calM')}$.
     Hence $\grob'$ is a Gr\"obner basis of $I_{\Sigma(\calM')}$.
     Thus, the initial terms in the Gr\"obner bases $\grob$ for $I_{\Sigma(\calM)}$ and $\grob'$ for
     $I_{\Sigma(\calM')}$
     are the same if we identify $x_{a_i}$ with $x_{a_i'}$, and $y_{a_i}$ with $y_{a_i'}$.
     Hence there exist regular unimodular triangulations $\Delta$ and $\Delta'$ of $\SEP(\calM)$ and $\SEP(\calM')$, respectively, having the same combinatorial $h$-polynomials $h(\Delta;t)=h(\Delta';t)$.     
     Hence $h^*(\Sigma(\calM); t) =
     h(\Delta;t)=h(\Delta';t)= h^*(\Sigma(\calM'); t)$.
     Thus
     \[
        h^*(\Sigma(\calM); t) = h^*(\Sigma(\calM'); t) = h^*(\Sigma( a_1'
        ); t)h^*(\Sigma(\calM/e); t) = (1+t)h^*(\Sigma(\calM/e); t),
     \]
     as claimed.
\end{proof}



Next we present an extension of \cite[Theorem 44]{DDM}.
Its proof is entirely analogous to the proof of \cite[Theorem 44]{DDM}, in the same way that our proof of Proposition~\ref{prop: contraction} adapts the proof of \cite[Proposition 5.4]{OhsugiTsuchiya21}.

\begin{thm}[{cf. \cite[Theorem 44]{DDM}}]
    Let $\calM_1$ be a bipartite regular matroid and let $\calM_2$ be a regular matroid.
    Then
    \[
        h^*(\SEP(P(\calM_1,\calM_2));t) = \frac{h^*(\SEP(\calM_1);t)h^*(\SEP(\calM_2);t)}{1 + t}.
    \]
\end{thm}

\begin{proof}
Let $e_0$ be the basepoint and let $\calM = P(\calM_1,\calM_2)$.
Since $\calM_1$ is bipartite, $e_0$ is not a loop of $\calM_1$.

\bigskip

\noindent
{\bf Case 1.}
($e_0$ is a coloop of $\calM_1$)
A coloop in $\calM_1$ is a loop in $\calM_1^*$, and therefore $\calM_1 - e_0 = \calM_1 / e_0$. 
Thus
$\calM = (\calM_1-e_0) \oplus \calM_2=
(\calM_1/e_0) \oplus \calM_2$.
From Corollary \ref{free_sum} and Proposition \ref{prop: contraction},
we have 
\[
        h^*(\SEP(\calM);t) =
        h^*(\SEP(\calM_1/e_0);t)h^*(\SEP(\calM_2);t)
        =\frac{h^*(\SEP(\calM_1);t)}{1+t}h^*(\SEP(\calM_2);t).
      \]

\noindent
{\bf Case 2.}
($e_0$ is not a coloop of $\calM_1$)
By definition,
\[
    \calC(\calM) = \, \calC(\calM_1) \cup \calC(\calM_2) \cup \{(C_1 \cup C_2) - e_0 \mid e_0 \in C_i \in \calC(\calM_i) \text{ for each } i\}.
\]
From Lemma~\ref{lem: initial terms},
    the collection of all binomials of the following types forms a Gr\"obner basis of $I_{\Sigma(\calM)}$ with respect to the grevlex order induced from $z < x_{e_0} < y_{e_0} < \cdots$.
    \begin{enumerate}[label={\rm (\roman*)}]
        \item 
        For every $2k$-element circuit $C$ of $\calM_1$ and choice of orientation for each $e \in C$, and any $k$-element subset $I \subseteq C$ not containing the weakest element among those of $C$,
        \[
            \prod_{e \in I} p_e - \prod_{e \in C - I} q_e.
        \]
        \item 
        For every $2k$-element circuit $C$ of $\calM_2$ and choice of orientation for each $e \in C$, and any $k$-element subset $I \subseteq C$ not containing the weakest element among those of $C$,
        \[
            \prod_{e \in I} p_e - \prod_{e \in C - I} q_e.
        \]
        \item 
        For every $(2k+1)$-element circuit $C$ of $\calM_2$ and choice of orientation for each $e \in C$, and any $(k+1)$-element subset $I \subseteq C$,
        \[
            \prod_{e \in I} p_e - z\prod_{e \in C - I} q_e.
        \]
        \item 
        For any $e \in E$, $x_ey_e - z^2$.
        
        \item         
        For every $2k$-element circuit $C=(C_1 \cup C_2)-e_0$ of $\calM$ where 
        $e_0 \in C_i \in \calC(\calM_i)$ for each $i$
        and choice of orientation for each $e \in C$, and any $k$-element subset $I \subseteq C$ not containing the weakest element among those of $C$,
        \[
            \prod_{e \in I} p_e - \prod_{e \in C - I} q_e.
        \]
        In this case, since $|C_1|$ is even, so is $|C_2|$.

        \item
        For every $(2k+1)$-element circuit $C=(C_1 \cup C_2)-e_0$ of $\calM$ where 
        $e_0 \in C_i \in \calC(\calM_i)$ for each $i$
        and choice of orientation for each $e \in C$, and any $(k+1)$-element subset $I \subseteq C$,
        \[
            \prod_{e \in I} p_e - z\prod_{e \in C - I} q_e.
        \]
        In this case, since $|C_1|$ is even, $|C_2|$ is odd.
    \end{enumerate}
    
Let 
\[
b=\prod_{e \in I} p_e - \prod_{e \in C - I} q_e
\]
be a binomial of the form in (v).
Since $|I| = |C-I| =k$, either 
 $|(C_1-e_0) \cap I| \ge |(C_1-e_0) - I|$
 or
 $|(C_2-e_0) \cap I| \ge |(C_2-e_0) - I|$ holds.
Suppose that $|(C_i-e_0) \cap I| \ge |(C_i-e_0) - I|$
for $1 \le i \le 2$.
Since $|C_i-e_0|$ is odd, 
 $|(C_i-e_0) \cap I| > |(C_i-e_0) - I|$,
 and hence $|C_i \cap I| \ge | C_i - I|$.
Let 
 \[b'= \prod_{e \in I'} p_e - \prod_{e \in C_i - I'} q_e,\]
 where $I' \subseteq I$ and $|C_i \cap I'| = | C_i - I'|$.
 Then $b'$
 is a binomial of the form in (i) or (ii), and the initial term of $b$ is divisible by the initial term of $b'$.

Let 
\[
b=\prod_{e \in I} p_e - z \prod_{e \in C - I} q_e
\]
be a binomial of the form in (vi).
Since $|I| = k+1$ and $|C-I| =k$, either 
 $|(C_1-e_0) \cap I| > |(C_1-e_0) - I|$
 or
 $|(C_2-e_0) \cap I| > |(C_2-e_0) - I|$ holds.
Then 
either 
 $|C_1 \cap I| \ge |C_1 - I|$
 or $|C_2 \cap I| \ge |C_2 - I|$
 holds.
If $|C_1 \cap I| \ge |C_1 - I|$,
then,
by the same argument above,
we can find a binomial $b'$ in (i) such that
the initial term of $b$ is divisible by the initial term of $b'$.
Suppose that $|C_2 \cap I| \ge |C_2 - I|$.
Since $|C_2|$ is odd, we have $|C_2 \cap I| > |C_2 - I|$.
Let 
 \[b'= \prod_{e \in I'} p_e - z \prod_{e \in C_2 - I'} q_e,\]
 where $I' \subseteq I$ and $|C_2 \cap I'| = | C_2 - I'|+1$.
 Then $b'$
 is a binomial of the form in (iii),
 and the initial term of $b$ is divisible by the initial term of $b'$.

Thus the binomials in (v) and (vi) are redundant, and the collection of all binomials in (i), (ii), (iii), and (iv)
above forms a Gr\"obner basis of $I_{\Sigma(\calM)}$ with respect to the grevlex order.
Hence the initial ideal of $I_{\Sigma(\calM)}$ is generated by the initial terms of binomials in (i), (ii), (iii), and (iv).
Moreover,
\begin{itemize}
    \item[(1)]
    the variables $x_{e_0}$, $y_{e_0}$, and $z$ do not appear in the initial terms of  binomials in (i) arising from $\calM_1$, and
    \item[(2)]
    there are no common variables in the initial terms of the binomials in (i) arising from $\calM_1$ and that in (ii) and (iii) arising from $\calM_2$.
\end{itemize}
%
Hence, from (1), (2) and \cite[Lemma 43 and Proof of Theorem 44]{DDM}, we have 
     \[
\frac{h^*(\SEP(\calM);t)}{(1-t)^{{\rm rank}(\calM)+1}}
    = \frac{h^*(\SEP(\calM_1);t)}{(1-t)^{{\rm rank}(\calM_1)+1}} \frac{h^*(\SEP(\calM_2);t)}{(1-t)^{{\rm rank}(\calM_2)+1}}\frac{(1 - t)^2}{1 + t}
   = \frac{h^*(\SEP(\calM_1);t)h^*(\SEP(\calM_2);t)}{(1-t)^{{\rm rank}(\calM)+1} (1 + t)}.
    \]
     Multiplying both sides by $(1 - t)^{{\rm rank}(\calM)+1}$ yields the claim.
\end{proof}

\section{A counterexample of $\gamma$-nonnegativity for generalized SEPs}\label{sec: counterexample}

In \cite{OhsugiTsuchiya21}, the authors conjectured that SEPs are $\gamma$-nonnegative. 
It is natural to want to extend this conjecture to generalized SEPs, but the conjecture in this setting is false.
Throughout the remainder of this section, we will be considering the matroids $M^*(K_{3,n})$ for $n \geq 3$ where $K_{3,n}$ is a complete bipartite graph.

\begin{prop}\label{prop:gamma}
Let $\calM=M^*(K_{3,n})$ with $n \geq 3$ and let $E$ be the ground set of $\calM$. 
\begin{enumerate}[label={\rm (\roman*)}]
    \item $\calM$ can be represented by $[I_{2(n-1)} \mid B_n]$, where $B_n$ is the $2(n-1) \times (n+2)$ matrix \[
    B_n = 
    \begin{bmatrix}
        1 & 1 & 0 & -1 & 0 & 0 & \cdots & 0 \\
        1 & 0 & 1 & -1 & 0 & 0 & \cdots & 0 \\
        1 & 1 & 0 & 0 & -1 & 0 & \cdots & 0 \\
        1 & 0 & 1 & 0 & -1 & 0 & \cdots & 0 \\
        1 & 1 & 0 & 0 & 0 & -1 & \cdots & 0 \\
        1 & 0 & 1 & 0 & 0 & -1 & \cdots & 0 \\
        \vdots & \vdots & \vdots & \vdots & \vdots & \vdots & \ddots & \vdots \\
        1 & 1 & 0 & 0 & 0 & 0 & \cdots & -1 \\
        1 & 0 & 1 & 0 & 0 & 0 & \cdots & -1 \\
    \end{bmatrix}. 
\]
    \item $\calM - e$ is not graphic for any $e \in E$ if $n \geq 4$. 
\end{enumerate}
\end{prop}
\begin{proof}
    \hfill
    \begin{enumerate}[label=(\roman*)]
        \item  It is known that if a matroid $\calM$ is represented by a matrix $A = [I \mid D]$, then $\calM^*$ is represented by $[-D^T \mid I]$, where $\cdot^T$ denotes the transpose (\cite[Theorem 2.2.8]{Oxley}). 
        When permuting columns of $[-D^T \mid I]$, the resulting matrices represent matroids isomorphic to $\calM^*$, and therefore may also be taken as matrices representing $\calM^*$. 
        Hence, we prove that the matrix $[I_{n+2} \mid -D^T]$ represents a matroid isomorphic to $M(K_{3,n})$ where $D= B_n$.
        
        Let $V(K_{3,n})=\{u_1,u_2,u_3\} \cup \{v_1,\ldots,v_n\}$ and $E(K_{3,n})=\{u_iv_j \mid 1 \leq i \leq 3, 1 \leq j \leq n\}$. 
    Fix a spanning tree consisting of the following edges: 
    \[
        e_1=u_1v_1, \quad e_2=v_1u_2,\quad e_3=v_1u_3,\quad e_4=u_1v_2,\,\ldots\,,e_{n+2}=u_1v_n.
    \]
    This corresponds to the identity matrix in the representing matrix $[I \mid D']$ of $M(K_{3,n})$. 
    Our goal is to show that $D'=-D^T =-B_n^T$. 
    By assigning the orientation of each of these edges as described (i.e., from the first vertex to the second vertex), 
    for the remaining edges of $K_{3,n}$, we see that $u_2v_j$ (resp. $u_3v_j$) with $2 \leq j \leq n$ forms minimal dependent sets with the opposite of $e_2$ (resp.~the opposite of $e_3$), the opposite of $e_1$, and $e_{j+2}$. 
    This dependent set corresponds to the $(2j-1)$-th (resp. $2j$-th) column of $-D^T$, as claimed.
    
    \item It is known that given a matroid $\calM$ and an element $e$ of its ground set, we have $\calM^*- e = (\calM/e)^*$ (see \cite[3.1.1]{Oxley}). 
    Hence, it is enough to show that $K_{3,n}/e$ never becomes planar for any edge of $K_{3,n}$.
    This is easy to see: the edge set of $K_{3,n}/u_1v_1$ is $\{u_iv_j \mid i=2,3, \, 1 \leq j \leq n\} \cup \{v_1v_k \mid 2 \leq k \leq n\}$, which contains $K_{3,n-1}$ as a subgraph. \qedhere
    \end{enumerate}
\end{proof}



A helpful consequence of Corollary~\ref{free_sum} on the side of $\gamma$-vectors is the following: writing $\gamma(P) = (\gamma_0,\dots,\gamma_{\lfloor d/2 \rfloor})$, we have
\[
    h^*(P \oplus [-1,1]; t) = h^*(P;t)(1+t) = \sum_{i=0}^{\lfloor d/2\rfloor} \gamma_it^i(1+t)^{d+1-2i}.
\]
That is, $\gamma(P \oplus [-1, 1]) = \gamma(P)$.
This will be a crucial observation in our next result.

\begin{thm}\label{thm: not nonnegative}
    In every dimension at least $10$ there is a generalized SEP that is not $\gamma$-nonnegative.
\end{thm}

\begin{proof}
    Consider the matroid $\calM = M^*(K_{3,6})$.
    Since $K_{3,6}$ has $18$ edges, we know that $\dim \Sigma(M^*(K_{3,6}))=18-(9-1)=10$. 
    From its matrix description, it is possible to show that 
    \[
        h^*(\Sigma(\calM)) = (1, 26, 297, 1908, 6264, 9108, 6264, 1908, 297, 26, 1),
    \]
    hence $\gamma(\Sigma(\calM)) = (1, 16, 124, 596, 914, -148)$.
    We computed the $h^*$-vector for this particular example through a combination of the software Macaulay2 \cite{M2} and polymake \cite{polymake}. 
    
    Now let $\calM_k = \calM \oplus x_1 \oplus \cdots \oplus x_k$ where the $x_i$ are pairwise distinct elements, none of which are already in $\calM$.
    Then $\dim(\SEP(\calM_k)) = \rk(\calM_k)=
    \rk (\calM) + \rk(x_1) + \cdots + \rk(x_k) =10+k$. 
    By using Corollary~\ref{free_sum}, we have
    \[
        h^*(\Sigma(\calM_k); t) = h^*(\SEP(\calM); t) h^*(\SEP(x_1); t) \cdots h^*(\SEP(x_k); t) = 
        h^*(\SEP(\calM); t)(1+t)^k,
    \]
    from which it follows that $\gamma(\Sigma(\calM_k)) = \gamma(\Sigma(\calM))$.
\end{proof}

Computational evidence suggests that the $\gamma$-vectors for $\Sigma(M^*(K_{3,n}))$ fail to be nonnegative for all even $n \geq 6$, although we do not aim to prove so here.
Instead, we examine how these matroids are situated between ordinary SEPs and generalized SEPs: by deleting the two edges $e$ and $e'$ that are incident to the same degree-$3$ vertex of $K_{3,n}$
from $M^*(K_{3,n})$, the resulting matroid is now graphic.
In fact, by invoking Whitney's $2$-isomorphism theorem \cite[Theorem 5.3.1]{Oxley}, it is a brief exercise to show that there is, up to isomorphism, a unique graph $\Gamma(n)$ such that $M^*(K_{3,n}) - \{e,e'\} = M(\Gamma(n))$.
We will be more precise about what this graph looks like in Section~\ref{sec: proof}.
For now, we simply state our main results regarding $\Sigma(\Gamma(n))$.


\begin{thm}\label{thm: h star of gamma}
    For each $n \geq 1$,
    \begin{equation}\label{eq:formula}   
        h^*(\Sigma(\Gamma(n+1));t) = \sum_{\ell=0}^{\lfloor n/2 \rfloor}(2t)^{n-2\ell-1}\left( 2t\binom{n}{2\ell} + (1+t)^2\binom{n}{2\ell+1} \right) \sum_{(p,q) \in S_{2\ell}}\binom{2\ell+1}{p+q+1}f_{p,q,\ell}(t),
    \end{equation}
        where 
    \[
        S_k =\{(x,y) \in \ZZ^2 \mid x \geq 0, y \geq 0, x+y \leq k\}
    \]
    and
    \[
        f_{p,q,\ell}(t) = t^{2\ell-p-q}\sum_{i=0}^p\sum_{j=0}^q\binom{p}{i}\binom{q}{j}\binom{2\ell-p-q}{\ell-q-i+j}t^{2(i+j)}. 
    \]
\end{thm}


For example, we can compute $h^*(\Sigma(\Gamma(3));t)$ as 
\[\begin{aligned}
 h^*(\Sigma(\Gamma(3));t) &= (2t)^1\left(2t\binom{2}{0} + (1+t)^2\binom{2}{1}\right)\binom{1}{1}f_{0,0,0}(t) \\
 &{} \qquad + (2t)^{-1}\left(2t\binom{2}{2} + (1+t)^2\binom{2}{3}\right)\sum_{(p,q) \in S_2} \binom{3}{p+q+1}f_{p,q,1}(t),
\end{aligned}
\]
which simplifies to 
\[
    h^*(\Sigma(\Gamma(3));t) = 1 + 10t + 22t^2 + 10t^3 + t^4.
\]

Our second main result regarding $\Gamma(n)$ is the following. 
For readability, we defer its proof to Section~\ref{sec: proof2}.

\begin{thm}\label{thm: gamma vec}
    For each $n \geq 1$,
    \begin{align*}
        \gamma(\Sigma(\Gamma(n+1));t)&=\sum_{\ell=0}^{\lfloor n/2 \rfloor}(2t)^{n-2\ell-1}\left(2\binom{n}{2\ell}t+\binom{n}{2\ell+1}\right)\sum_{a=0}^\ell\binom{2a}{a}t^a\\
            &=    \sum_{m = 0}^{n}
    \binom{n}{m}
    (2t)^{n-m} 
    \sum_{a=0}^{\lfloor m/2 \rfloor} \binom{2a}{a} t^a.
    \end{align*}
    In particular, $\Sigma(\Gamma(n+1))$ is $\gamma$-nonnegative for all $n$.
\end{thm}

As a corollary of Theorem~\ref{thm: h star of gamma}, we can obtain the normalized volume of $\Sigma(\Gamma(n+1))$. 
For its proof, as well as the proof of Theorem~\ref{thm: h star of gamma}, 
we collect some well-known formulas on the summation of binomial coefficients.
For reference, see, for example, \cite[Equations (5.22), (5.26), (5.21)]{GrahamKnuthPatashnik}.
\begin{lem}\label{lem:binom}
    Let $a,b,c,d \in \ZZ_{\geq 0}$. Then the following hold{\rm :} 
    \begin{enumerate}[label={\rm (\roman*)}]
        \item $\sum_{i=0}^c \binom{a}{i}\binom{b}{c-i}=\binom{a+b}{c}${\rm ;}
        \item When $b \geq a$, $\sum_{i=0}^c \binom{a+i}{b}\binom{c-i}{d}=\binom{a+c+1}{b+d+1}${\rm ;}
        \item $\binom{a}{b}\binom{b}{c}=\binom{a-c}{b-c}\binom{a}{c}$.
    \end{enumerate}
\end{lem}

\begin{cor}\label{cor: normalized volume}
    The normalized volume of $\Sigma(\Gamma(n+1))$ is
    \[
        2^n \sum_{k=0}^n (k+1)\binom{n}{k}\binom{k}{\lfloor\frac{k}{2}\rfloor}.
    \]
\end{cor}
\begin{proof}
    We may set $t=1$ in \eqref{eq:formula}.
    For $f_{p,q,\ell}(1)$, we see that 
    \[f_{p,q,\ell}(1)=\sum_{i=0}^p\sum_{j=0}^q\binom{p}{i}\binom{q}{j}\binom{2\ell-p-q}{\ell-q-i+j} 
    =\sum_{j=0}^q\binom{q}{j}\binom{2\ell-q}{\ell-q+j}=\sum_{j=0}^q\binom{q}{j}\binom{2\ell-q}{\ell-j}=\binom{2\ell}{\ell}\]
    by using Lemma~\ref{lem:binom} (i) twice.
    Hence, 
    \[\sum_{(p,q) \in S_{2\ell}}\binom{2\ell+1}{p+q+1}\binom{2\ell}{\ell}=\binom{2\ell}{\ell}\sum_{m=0}^{2\ell}(m+1)\binom{2\ell+1}{m+1}=\binom{2\ell}{\ell}\sum_{m=0}^{2\ell}(2\ell+1)\binom{2\ell}{m}=2^{2\ell}(2\ell+1)\binom{2\ell}{\ell}.\]
    Therefore, 
    \begin{align*}
    h^*(\Sigma(\Gamma(n+1)); 1)&=\sum_{\ell=0}^{\lfloor n/2 \rfloor}2^{n-2\ell}\left( \binom{n}{2\ell} + 2\binom{n}{2\ell+1} \right) \sum_{(p,q) \in S_{2\ell}}\binom{2\ell+1}{p+q+1}\binom{2\ell}{\ell} \\
    &=\sum_{\ell=0}^{\lfloor n/2 \rfloor}2^{n-2\ell}\left( \binom{n}{2\ell} + 2\binom{n}{2\ell+1} \right) 2^{2\ell}(2\ell+1)\binom{2\ell}{\ell} \\
    &=2^n \left(\sum_{\ell=0}^{\lfloor n/2 \rfloor}(2\ell+1)\binom{n}{2\ell}\binom{2\ell}{\ell} + 2\sum_{\ell=0}^{\lfloor n/2 \rfloor}(2\ell+1)\binom{n}{2\ell+1}\binom{2\ell}{\ell} \right) \\
    &=2^n \left(\sum_{\ell=0}^{\lfloor n/2 \rfloor}(2\ell+1)\binom{n}{2\ell}\binom{2\ell}{\ell} + \sum_{\ell=0}^{\lfloor n/2 \rfloor}(2\ell+2)\binom{n}{2\ell+1}\binom{2\ell+1}{\ell} \right) \\
    &=2^n \sum_{k=0}^n (k+1)\binom{n}{k}\binom{k}{\lfloor\frac{k}{2}\rfloor}. \qedhere
    \end{align*}
\end{proof}



\section{Proof of Theorem~\ref{thm: h star of gamma}}\label{sec: proof}

In this section, we put all of the necessary pieces together for the proof of Theorem~\ref{thm: h star of gamma}. 
We will produce a triangulating tree set by identifying which oriented spanning trees correspond to monomials that are not divisible by an initial monomial of one of the types in Lemma~\ref{lem: initial terms}.
Before this, we collect several identities needed in the proof of Theorem~\ref{thm: h star of gamma}.

\begin{lem}\label{lem: breaking up coefficient}
    Let $n > 0$ and $p,q,\ell \geq 0$ such that $p+q \leq 2\ell \leq n$.
    \begin{enumerate}[label={\rm (\roman*)}]
        \item If $q \geq 1$, then
        \[
            \binom{n-(p+q)}{2\ell-(p+q)}\sum_{i=0}^{n-1} \binom{n-1-i}{p}\binom{i}{q-1} + \binom{n-1-(p+q)}{2\ell-1-(p+q)}\sum_{i=0}^{n-1} \binom{n-1-i}{p}\binom{i}{q}
         \]
         is equal to
        \[
            \binom{n}{2\ell}\binom{2\ell + 1}{p+q + 1}.
       \]
        \item For all $n,p,k,\ell \geq 0$ such that $p+q \leq 2\ell \leq n-1$,
        \[
            \binom{n-1-(p+q)}{2\ell-(p+q)}\sum_{i=0}^{n-1} \binom{n-1-i}{p}\binom{i}{q}= \binom{n}{2\ell+1}\binom{2\ell+1}{p+q+1}.
        \]
    \end{enumerate}
\end{lem}
\begin{proof}
    The formula in (i) follows from applying Lemma~\ref{lem:binom} (ii) to each sum.
    For the first sum, by using 
    $a = 0$, $b=q-1\ (\ge 0=a)$, $c=n-1\ge 0$, and $d=p \ge 0$ in Lemma~\ref{lem:binom} (ii),
    we have
    \[
    \sum_{i=0}^{n-1} \binom{n-1-i}{p}\binom{i}{q-1}=
    \binom{n}{p+q}.
    \]   
    For the second sum, by using 
    $a = 0$, $b=q$, $c=n-1$, and $d=p$ in Lemma~\ref{lem:binom} (ii),
    we have
    \[
    \sum_{i=0}^{n-1} \binom{n-1-i}{p}\binom{i}{q}=
    \binom{n}{p+q+1}.
    \]
    From these facts and Lemma \ref{lem:binom} (iii), 
    \[\begin{aligned}
        &\qquad \sum_{i=0}^{n-1} \binom{n-i-1}{p}\binom{i}{q-1}\binom{n-(p+q)}{2\ell-(p+q)} + \sum_{i=0}^{n-1} \binom{n-i-1}{p}\binom{i}{q}\binom{n-1-(p+q)}{2\ell-1-(p+q)} \\
        &= \binom{n-(p+q)}{2\ell-(p+q)}\binom{n}{p+q} + \binom{n-1-(p+q)}{2\ell-1-(p+q)}\binom{n}{p+q+1} \\
        &= \binom{n}{2\ell}\binom{2\ell}{p+q} + \binom{n}{2\ell}\binom{2\ell}{p+q+1} \\
        &= \binom{n}{2\ell}\binom{2\ell +1}{p+q+1}.
    \end{aligned}\]
    The formula in (ii) follows similarly.
\end{proof}

Now, note that $M^*(K_{3,n+1}) - e$ is never graphic for $n \geq 3$ (see Proposition~\ref{prop:gamma}). 
However, by an appropriate choice of $e,e'$ from $M^*(K_{3,n+1})$, namely where $e$ and $e'$ are edges incident to the same degree-$3$ vertex of $K_{3,n+1}$, we see that the matroid $M^*(K_{3,n+1}) - \{e,e'\}$ may be represented by the matrix $[I_{2n} \mid B'_{n+1}]$ where $B'_n$ is the matrix $B_n$ with the second and third columns removed. 
This matrix also represents the graph $\Gamma(n+1)$ on the ground set $\{u_1,v_1,\dots,u_n,v_n,u_{n+1}\}$ with edges $u_iu_{i+1}$, $u_iv_i$, $v_iu_{i+1}$ for each $i=1,\dots,n$, as well as $u_1u_{n+1}$. 
See Figure~\ref{fig: gammas} for examples.
\begin{figure}[ht]
    \includegraphics[scale=0.75]{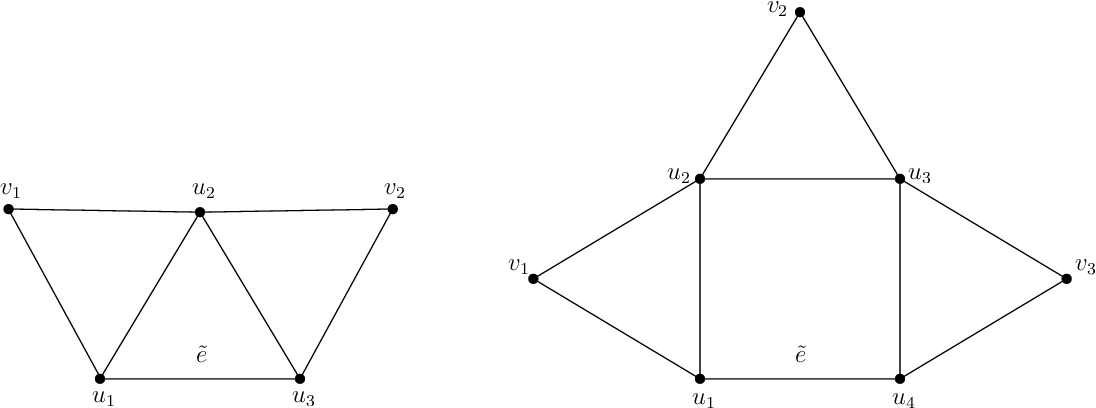}
    \caption{The graphs $\Gamma(3)$, left, and $\Gamma(4)$, right.}\label{fig: gammas}
\end{figure}

    Note that $\Gamma(n)$ is the same graph to which we alluded in the discussion preceding Theorem~\ref{thm: h star of gamma}.
    Indeed, although $K_{3,n}/e$ is not planar for any edge $e$ of $K_{3,n}$ as explained in the proof of Proposition~\ref{prop:gamma} (ii), we can get a planar graph $\tilde{G}$ by the contraction of an edge $e'$ of $K_{3,n}/e$ such that $e$ and $e'$ were incident to the same degree-$3$ vertex in $K_{3,n}$.
    By taking the dual graph of $\tilde{G}$ (in the sense of planar graphs), a planar graph $\Gamma(n)$ appears. 
    See Figure~\ref{fig: dual construct} for an illustration of $K_{3,4}$ with a choice of $e$ and $e'$ labeled, as well as a planar representation of $K_{3,4}/\{e,e'\}$ with its dual, $\Gamma(4)$, overlaid.

\begin{figure}
    \includegraphics[scale=1]{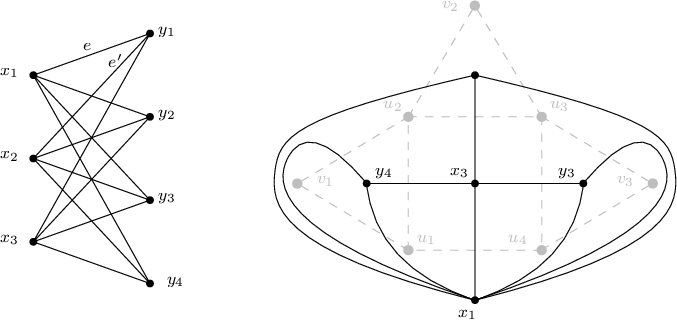}
    \caption{On the left, $K_{3,4}$ with a choice of $e$ and $e'$ labeled such that $\Gamma(4)$ is an underlying graph of $M^*(K_{3,4}) - \{e,e'\}$.
    On the right, a planar representation of $K_{3,4}/\{e,e'\}$ in solid black lines, with its dual, $\Gamma(4)$, overlaid in dashed gray lines.}\label{fig: dual construct}
\end{figure}
In what follows, we concentrate on the graph $\Gamma(n+1)$. 
Let $\etilde = u_1u_{n+1}$. 
To describe its spanning trees, it will be helpful to refer to a specific planar embedding of $\Gamma(n+1)$. 
We may view $\Gamma(n+1)$ as the $(2n+1)$-cycle $u_1v_1u_2v_2\dots u_{n+1}u_1$, together with the chords $u_iu_{i+1}$ for each $i = 1,\dots,n$.
In particular, throughout the remainder of the article, we will use ``chord'' to refer to any of the $n$ edges of the form $u_iu_{i+1}$.

Referring to a specific planar embedding will allow us to construct the $h^*$-polynomial by applying Theorem~\ref{thm: h-star contributions}.
To do this, we need to identify a triangulating tree set for $\Gamma(n+1)$.
Let $\calG_{n+1}$ be the reduced Gr\"obner basis of $I_{\Sigma(\Gamma(n+1))}$ with respect to any grevlex order in which the variable $z$ corresponding to the origin is weakest and the variables $p_{\etilde}$ and $q_{\etilde}$ corresponding to $\etilde$ are the second- and third-weakest.
Using the framework developed in Section~\ref{sec: background}, we can identify which oriented spanning trees belong to our triangulating tree set.
To help us describe them, we introduce some new terminology.

Observe that if $T$ is a spanning tree of $\Gamma(n+1)$, then for each of the triangles with vertices $u_i,v_i,u_{i+1}$, there are exactly three possibilities of edges appearing in $T$: $T$ contains either
\begin{itemize}
    \item both $u_iv_i$ and $v_iu_{i+1}$, which we call a \emph{chordless pair},
    \item both $u_iv_i$ and $u_iu_{i+1}$ or both $u_iu_{i+1}$ and $u_{i+1}v_i$, each of which we call a \emph{chorded pair}, or
    \item only $u_iv_i$ or only $v_iu_{i+1}$, each of which we call an \emph{unpaired edge}. 
\end{itemize}
For example, the spanning tree of $\Gamma(5)$ in Figure~\ref{fig: pairs} contains one chordless pair of edges, two chorded pairs of edges, and one unpaired edge. 

\begin{figure}[h]
    \centering
    \includegraphics[scale=0.6]{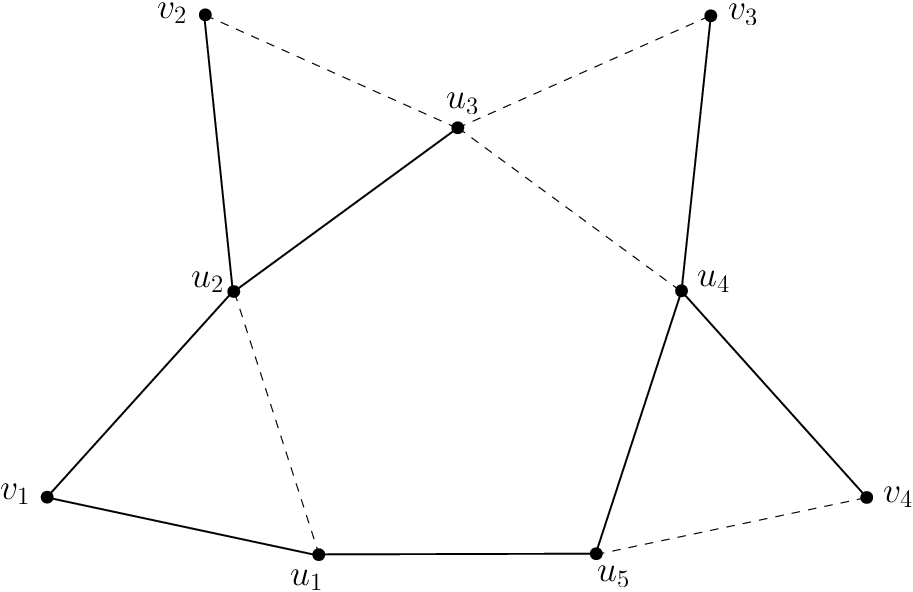}
    \caption{A spanning tree of $\Gamma(5)$, indicated by the solid edges. 
    The edges $u_1v_1$ and $v_1u_2$ form a chordless pair.
    The edges $u_2v_2$ and $u_2u_3$ form a chorded pair of type $\alpha$, while the edges $u_4v_4$ and $u_4u_5$ form a chorded pair of type $\beta$.
    The edge $v_3u_4$ is an unpaired edge.}
    \label{fig: pairs}
\end{figure}

\begin{prop}\label{prop: spanning tree characterization}
    An oriented spanning tree $T$ of $\Gamma(n+1)$ is part of the triangulating tree set if and only if, for every cycle $C$ in $\Gamma(n+1)$, either $T$ contains no more than $\lfloor\frac{|C|-1}{2}\rfloor$ edges of $C$ that are oriented in the same direction, or the following condition holds:
    \begin{itemize}
        \item[] When $|C|$ is even, both $\etilde \in T$ and $T$ contains exactly $|C|/2$ edges of $C$ that are oriented in the same direction as $\etilde$.
    \end{itemize}
\end{prop}

\begin{proof}
    This follows from Lemma~\ref{lem: initial terms}. 
\end{proof}

See Figure~\ref{fig: example trees} for  examples of spanning trees of $\Gamma(4)$ satisfying the conditions of Proposition~\ref{prop: spanning tree characterization}.

\begin{figure}[h]
    \centering
    \includegraphics[scale=0.7]{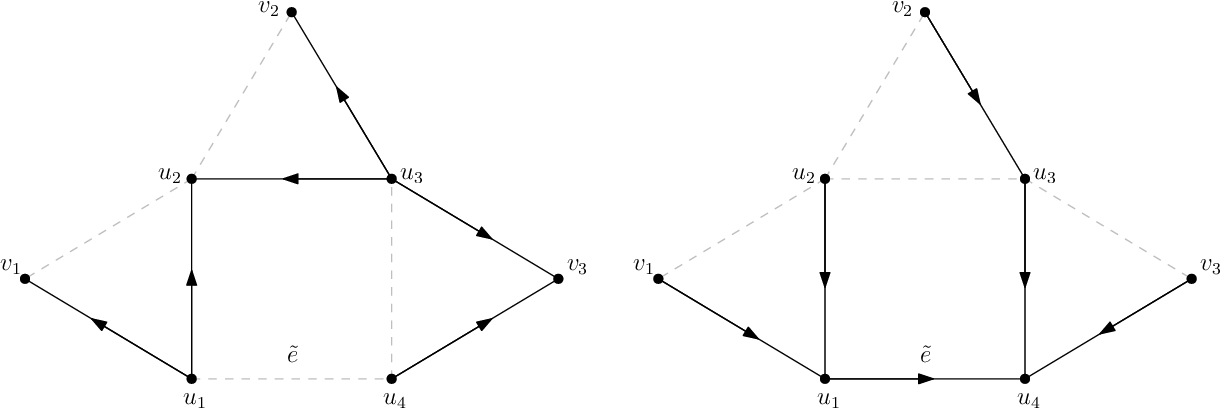}
    \caption{Two spanning trees of $\Gamma(4)$ that are members of the triangulating tree set, as described in Proposition~\ref{prop: spanning tree characterization}.}
    \label{fig: example trees}
\end{figure}

\begin{cor}\label{cor: chordless pair orientations}
    Suppose $T$ is an oriented spanning tree in the triangulating tree set. 
    Then the edges in any chordless pair of $T$ share the same tail or share the same head.
    Moreover, their orientations may be reversed to obtain another oriented spanning tree in the triangulating tree set. 
\end{cor}


\begin{proof}
If the edges in the chordless pair, say, $u_iv_i$ and $v_iu_{i+1}$, share neither the same tail nor the same head, then the $3$-cycle $u_i,v_i,u_{i+1}$ of $\Gamma(n+1)$ contains two edges that are oriented in the same direction.
\end{proof}

For a spanning tree $T$ of $\Gamma(n+1)$, let $c(T)$ denote the number of chords $u_iu_{i+1}$ of $T$ and let 
\[
    \epsilon(T) = 
        \begin{cases}
            1 & \text{ if } \etilde \in T \\
            0 & \text{ otherwise.}
        \end{cases}
\]

\begin{lem}\label{lem: feasible trees}
    Let $T$ be a member of our triangulating tree set. 
    \begin{enumerate}[label={\rm (\roman*)}]
        \item If $\epsilon(T)=1$, then $T$ contains a total of $n-1$ chordless pairs and chorded pairs of edges as well as exactly one unpaired edge. 
        \item If $\epsilon(T)=0$, then $T$ contains a total of $n$ chordless and chorded pairs of edges. Moreover, $c(T)$ must be always even in this case.  
    \end{enumerate}
\end{lem}
\begin{proof}
    Note that each spanning tree of $\Gamma(n+1)$ has $2n$ edges. Hence, the statements on the numbers of chordless/chorded pairs and the unpaired edge follow. 
    Regarding the evenness of $c(T)$, suppose to the contrary that
    $\epsilon(T)=0$ and $c(T)=2a-1$ for some $a \in \ZZ_{>0}$. 
    By the Pigeonhole Principle, there are at least $a$ chords of $T$ that have the same direction; call these chords $S_1$.
    On the other hand, since there are $n-2a+1$ chordless pairs which, together, contain $2(n-2a+1)$ edges, we know by 
    Corollary~\ref{cor: chordless pair orientations} that exactly half of these edges will have one orientation and the other half will have the opposite orientation.
    Let $S_2$ be the edges among the chordless pairs whose orientations are in the same direction as the edges of $S_1$.

    Set $S = S_1 \cup S_2$ and note that $|S| \ge n-a+1$.
    Construct a cycle $C$ in $\Gamma(n+1)$ consisting of $\etilde$, all chordless pairs of $T$, and all $c(T) = 2a - 1$ chords of $T$. 
    In particular, $S \subseteq C$.
    By design, $C$ has $1 + 2(n-2a+1) + (2a-1) =2(n-a+1)$ edges.
    Since $S \subseteq C$, there are at least $n-a+1$ edges in $C$ orientated in the same direction.
    This contradicts Proposition~\ref{prop: spanning tree characterization} since, in order for $T$ to be part of a triangulating tree set, $T$ would have had to contain no more than
    \[
        \left\lfloor\frac{|C|-1}{2}\right\rfloor = \left\lfloor\frac{2n-2a+1}{2}\right\rfloor = n - a
    \]
    edges oriented in the same direction.
\end{proof}

At this point we will take time here to provide several definitions that will help with the full proof, which becomes somewhat technical. 

First, suppose $c(T)=0$ and $\epsilon(T)=0$ and let $T$ be a member of our triangulating tree set satisfying both equations. 
We know, then, that the (undirected) edges of $T$ consist entirely of chordless pairs.
Observe that each pair of edges $u_iv_i$, $v_iu_{i+1}$ can be given a valid orientation in exactly two ways: either both edges point toward $v_i$ or both edges point away from $v_i$ (see Corollary~\ref{cor: chordless pair orientations}).
In both cases, there will be $n$ edges that point away from $u_1$.
So, by Theorem~\ref{thm: h-star contributions} with $v=u_1$, each such oriented tree contributes a summand of $t^n$ in $h^*(\Sigma(\Gamma(n+1)); t)$. 
There are $2^n$ valid orientations of this tree, resulting in a summand of $(2t)^n$ in the $h^*$-polynomial.

In subsequent cases, it will be helpful to refer to two portions of $T$ relative to $u_1$.

\begin{defn}
   Let $T$ be a spanning tree of $\Gamma(n+1)$. 
   The \emph{left side} of $T$, denoted $L(T)$, is the component of $T - \etilde$ containing $u_1$.
   The \emph{right side} of $T$, denoted $R(T)$, is the component of $T$ induced by $V(\Gamma(n+1)) - (V(L(T)) - u_1)$.
\end{defn}

In Figure~\ref{fig: pairs}, $L(T)$ is the subgraph of $T$ induced by $\{u_1,u_2,u_3,v_1,v_2\}$ and $R(T)$ is the subgraph induced by $\{u_1,u_4,u_5,v_3,v_4\}$.
Notice in particular that if $\etilde \notin T$, then $L(T) = T$ and $R(T)$
is the edgeless graph on $\{u_1\}$.

Now, suppose $c(T) = 0$ and $\epsilon(T) = 1$ and let $T$ be a member of our triangulating tree set satisfying both equations.
We can establish orientations on any chordless pair as in the previous case, although there are now only $n-1$ of these pairs.
For any cycle containing $\etilde$ and the unpaired edge $e$, the orientations of $\etilde$ and $e$ should also be opposite; otherwise we can find at least $(n+1)$ edges with the same direction in the $(2n+1)$-cycle of $\Gamma(n+1)$ (see Corollary~\ref{cor: chordless pair orientations} and Lemma~\ref{lem: feasible trees}). 

The definitions of $L(T)$ and $R(T)$ then make it clear that $T$ can contribute any of the following summands to the $h^*$-polynomial, depending on the orientation and the location of the unpaired edge, $e$:
\begin{itemize}
    \item If $\etilde$ and $e$ are pointing away from each other and $e \in L(T)$, then $T$ contributes the summand $t^2(2t)^{n-1}$.
    \item If $\etilde$ and $e$ are pointing away from each other and $e \in R(T)$, then $T$ contributes the summand $t(2t)^{n-1}$.
    \item If $\etilde$ and $e$ are pointing towards each other and $e \in L(T)$, then $T$ contributes the summand $(2t)^{n-1}$.
    \item If $\etilde$ and $e$ are pointing towards each other and $e \in R(T)$, then $T$ contributes the summand $t(2t)^{n-1}$.
\end{itemize}
Thus, by considering all orientations on the edges of a spanning tree satisfying $c(T) = 0$ and $\epsilon(T) = 1$, the tree will contribute a summand of $(1+t)^2(2t)^{n-1}$.

An additional subtlety arises when $c(T) = 2$ and $\epsilon(T) = 0$.
We can repeat the same argument for chordless pairs as when $c(T) =\epsilon(T)=0$, but we now have only $n-2$ chordless pairs.
The orientations on the chordless pairs result in a factor of $(2t)^{n-2}$ in the contribution $T$ makes to the $h^*$-polynomial.
Such a factor occurs $\binom{n}{2} = \binom{n}{c(T)}$ times, since this is the number of chords $u_iu_{i+1}$ that can appear as parts of chorded pairs. 

This time we must account for the orientations on the edges of each chorded pair in relation to the other. 
Again, for each chorded pair, the orientation on the edge $u_iu_{i+1}$ dictates the orientation of the other edge in the pair: by Proposition~\ref{prop: spanning tree characterization}, they have the opposite direction. 
However, whether the chorded pair contains $u_iv_i$ or $v_iu_{i+1}$ will make a difference in the summand of the $h^*$-polynomial that $T$ contributes.

\begin{defn}
    We say that a chorded pair has \emph{type $\alpha$} if the non-chord edge is incident to the vertex that is closer to $u_1$ within $T$; otherwise, we say that a chorded pair has \emph{type $\beta$}.
\end{defn}

See Figure~\ref{fig: pairs} for an example of chorded pairs of each type.

If a chorded pair of $T$ has type $\alpha$, then one of the orientations of its edges results in a factor of $t^2$ in the summand contributed to the $h^*$-polynomial, while the other orientation results in a factor of $1$.
On the other hand, if a chorded pair has type $\beta$, then each of the orientations of its edges results in a factor of $t$.
So, if we consider all spanning trees $T$ of $\Gamma(n+1)$ satisfying $c(T) = 2$ and $\epsilon(T) = 0$, then we know that there are $p$ chorded pairs of type $\alpha$ in $L(T)$ and $q$ chorded pairs of type $\alpha$ in $R(T)$ for some nonnegative $p,q$ satisfying $p + q \leq 2$.
This leaves $2-(p+q)$ chords to be part of a chorded pair of type $\beta$.
This is where the set $S_k$ in \eqref{eq:formula} will come from: the pair $(p, q)$ will indicate the presence of $p$ chorded pairs of type $\alpha$ in $L(T)$ and $q$ of them in $R(T)$.

The final case we examine before proving Theorem~\ref{thm: h star of gamma} is when $c(T) = 1$ and $\epsilon(T) = 1$.
Recognize here that in any oriented spanning tree $T$ of $\Gamma(n+1)$, there is at most one unpaired edge.
Consequently, if an unpaired edge exists, then $\etilde \in T$ and the unpaired edge must have endpoints $u_i$ and $v_j$ for some $i$ and $j \in \{i,i+1\}$. 
Thus we may, without affecting the monomial that $T$ contributes, modify $T$ by deleting the vertex $v_j$, adding a new vertex $v_{n+1}$, and adding a new edge as follows:
\begin{itemize}
    \item If $\etilde$ and $u_iv_j$ are pointing away from each other and $u_iv_j \in L(T)$, then add the directed edge $u_1v_{n+1}$.
    \item If $\etilde$ and $u_iv_j$ are pointing away from each other and $u_iv_j \in R(T)$, then add the directed edge $u_{n+1}v_{n+1}$.
    \item If $\etilde$ and $u_iv_j$ are pointing towards each other and $u_iv_j \in L(T)$, then add the directed edge $v_{n+1}u_1$.
    \item If $\etilde$ and $u_iv_j$ are pointing towards each other and $u_iv_j \in R(T)$, then add the directed edge $v_{n+1}u_{n+1}$.
\end{itemize}
See Figure~\ref{fig: modified trees} for illustrations of each of these cases.

\begin{defn}
    Let $T$ be a (oriented) spanning tree of $\Gamma(n+1)$.
    If $\epsilon(T) = 1$, then denote by $\md(T)$ the tree described in the paragraph above.
    Otherwise, we set $\md(T) = T$.
    We call $\md(T)$ a \emph{modified (oriented) spanning tree} of $\Gamma(n+1)$.
\end{defn}

Modified spanning trees allow us to view $\etilde$ and $u_1v_{n+1}$ or $u_{n+1}v_{n+1}$ as a chorded pair, so that all edges in $\md(T)$ are part of a chorded or chordless pair.
Hence, $c(\md(T)) = c(T) + \epsilon(T)$.
Note that the definitions of $L(T)$ and $R(T)$ can be extended to modified spanning trees as well.
In fact, the number of edges pointing towards or away from $u_1$ in $L(T)$ is the same as in $L(\md(T))$; the same is true for $R(T)$ and $R(\md(T))$.
We are now prepared to present the proof of Theorem~\ref{thm: h star of gamma}.

\begin{figure}
    \begin{center}
        \includegraphics[scale=0.7]{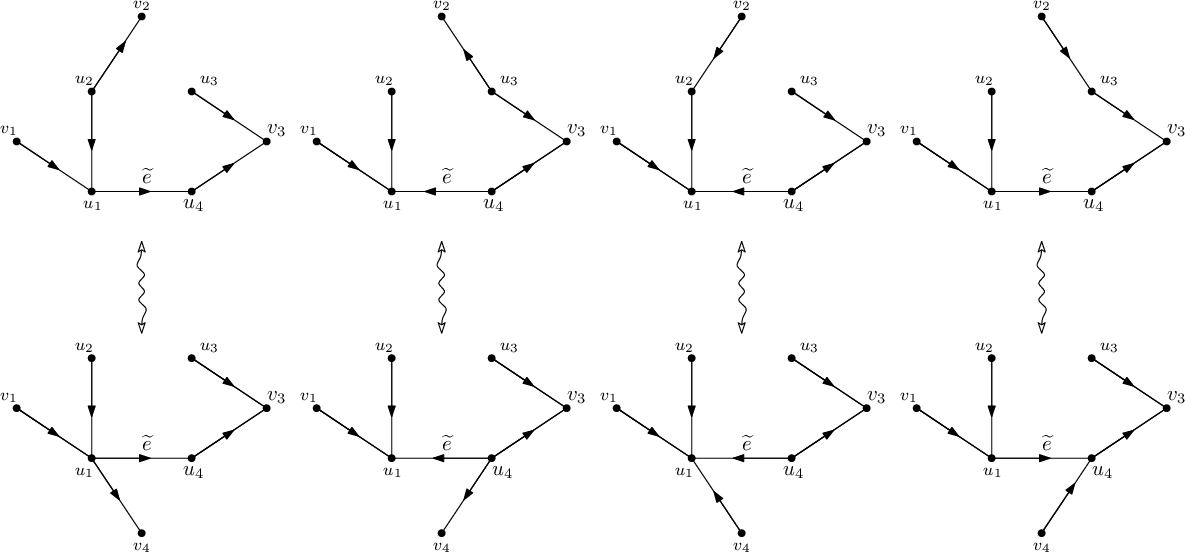}
    \end{center}
    \caption{Examples of modified spanning trees in $\Gamma(4)$.
    The original spanning trees and their corresponding modified trees are placed above and below each other.}\label{fig: modified trees}
\end{figure}



\begin{proof}[Proof of Theorem~\ref{thm: h star of gamma}]
    Fix an arbitrary $\ell$.
    We will treat the summand corresponding to $\ell$ in two pieces:
    \begin{equation}\label{eqn: proof first part}
        \sum_{(p,q) \in S_{2\ell}} (2t)^{n-2\ell}\binom{n}{2\ell}\binom{2\ell+1}{p+q+1}f_{p,q,\ell}(t)
    \end{equation}
    and
    \begin{equation}\label{eqn: proof second part}
        \sum_{(p,q) \in S_{2\ell}} (2t)^{n-2\ell-1}(1+t)^2\binom{n}{2\ell+1}\binom{2\ell+1}{p+q+1}f_{p,q,\ell}(t)
    \end{equation}
    since these will more closely resemble how we will identify how each oriented spanning tree in our triangulating tree set will contribute to the $h^*$-polynomial of $\Sigma(\Gamma(n+1))$.

    Importantly, it will initially be more convenient to enumerate the spanning trees in the triangulating tree set not by constructing each such $T$ directly, but by constructing the corresponding $\md(T)$.
    Namely, we first will collect all $\md(T)$ having $2\ell$ chords for which $T$ is in our triangulating tree set, some of which contain $\etilde$ and some of which do not.
    We will then see that the set of corresponding (unmodified) spanning trees will provide a total contribution of \eqref{eqn: proof first part} to the $h^*$-polynomial of $\Sigma(\Gamma(n+1))$.
    Then, we will enumerate the remaining modified spanning trees, meaning those satisfying $c(\md(T)) = 2\ell+1$, and show that the set of corresponding (unmodified) spanning trees provide a total contribution of \eqref{eqn: proof second part} to the $h^*$-polynomial of $\Sigma(\Gamma(n+1))$.
    Throughout, we use $p$ to denote the number of chorded pairs in 
    $L(\md(T))$ of type $\alpha$ and we use $q$ to denote the number of chorded pairs in 
    $R(\md(T))$ of type $\alpha$.

    First consider \eqref{eqn: proof first part}, that is, suppose $c(\md(T)) = 2\ell$.
    We will proceed by considering the following two cases. \\
    
    \noindent \textbf{Case 1: Construct all modified spanning trees for which $q \geq 1$.}
    In order to have $q \geq 1$, $R(\md(T))$ must be nontrivial, and so $\epsilon(T)=1$. 
    Also, there will be exactly one $i$, with $1 \leq i \leq n-1$, such that all of the edges $u_{n-i}u_{n-i+1}$, $u_{n-i}v_{n-i}$ and $v_{n-i}u_{n-i+1}$ are missing in $\md(T)$. 
    
    Fix an arbitrary $i$ and suppose first that $\etilde$ is part of a chorded pair of type $\alpha$.
    Necessarily, this pair will be in $R(\md(T))$.
    To construct the possible modified spanning trees with these properties, we may select $p$ of the chords $u_1u_2,\dots,u_{n-i-1}u_{n-i}$ to be edges in chorded pairs of type $\alpha$; these will all be part of $L(\md(T))$.
    There are $\binom{n-i-1}{p}$ such choices.
    We may then choose $q-1$ of the chords $u_{n-i+1}u_{n-i+2},\dots,u_nu_{n+1}$ to be part of chorded pairs of type $\alpha$, to ensure that $R(\md(T))$ contains a total of $q$ of them.
    There are $\binom{i}{q-1}$ such choices.
    
    Of the remaining $n-(p+q)$ pairs of vertices $u_j,u_{j+1}$ (which excludes $u_{n-i},u_{n-i+1}$), we must choose $2\ell-(p+q)$ of them to be part of chorded pairs of type $\beta$.
    So, we have constructed
    \begin{equation*}
        \sum_{i=0}^{n-1} \binom{n-i-1}{p}\binom{i}{q-1}\binom{n-(p+q)}{2\ell-(p+q)} 
    \end{equation*}
    trees.
    On the other hand, suppose $\etilde$ is part of a chorded pair of type $\beta$. Following an analogous argument, we obtain
    \begin{equation*}
        \sum_{i=0}^{n-1} \binom{n-i-1}{p}\binom{i}{q}\binom{n-1-(p+q)}{2\ell-1-(p+q)} 
    \end{equation*}
    modified spanning trees.
    Add these and apply 
    Lemma \ref{lem: breaking up coefficient} (i)
    to obtain
\[\sum_{i=0}^{n-1} \binom{n-i-1}{p}\binom{i}{q-1}\binom{n-(p+q)}{2\ell-(p+q)} + \sum_{i=0}^{n-1} \binom{n-i-1}{p}\binom{i}{q}\binom{n-1-(p+q)}{2\ell-1-(p+q)}
= \binom{n}{2\ell}\binom{2\ell +1}{p+q+1}
\]
    modified spanning trees in this case. \\

    \noindent \textbf{Case 2: Construct all modified spanning trees for which $q=0$.}
    This time, we proceed as in Case $1$, except there is no need to identify chords among $u_{n-i+1}u_{n-i+2},\dots,u_nu_{n+1}$ to be part of chorded pairs of type $\alpha$ in $R(\md(T))$.

    Suppose $\etilde$ is part of a chorded pair of type $\beta$ in $\md(T)$.
    Proceeding as before and using Lemma~\ref{lem:binom} (ii) and (iii), we obtain 
    \begin{eqnarray*}       
        \sum_{i=0}^{n-1} \binom{n-1-i}{p}\binom{n-1-p}{2\ell-1-p} &= &
        \binom{n-1-p}{2\ell-1-p} \sum_{i=0}^{n-1} \binom{0+i}{0} \binom{n-1-i}{p}\\
        &=& \binom{n-1-p}{2\ell-1-p} \binom{n}{p+1}\\
        &=& \binom{n}{2\ell}\binom{2\ell}{p+1}
    \end{eqnarray*}
    such trees.
    On the other hand, if $\etilde$ is not in $\md(T)$, then $R(\md(T))$ is the edgeless graph on $\{u_1\}$.
    There are $n$ edges $u_iu_{i+1}$ to select to be part of the chorded pairs, and of those, $p$ are chosen to be of type $\alpha$ and the rest must be of type $\beta$.
    In total, there are 
    \[
        \binom{n}{2\ell}\binom{2\ell}{p}
    \]
    modified spanning trees in this case. 
    All together, there are 
    \[
        \binom{n}{2\ell}\binom{2\ell}{p+1} + \binom{n}{2\ell}\binom{2\ell}{p} = \binom{n}{2\ell}\binom{2\ell+1}{p + q + 1}
    \]
    modified spanning trees in which $q=0$. 
    \\

    Since the numbers of modified spanning trees in Cases 1 and 2 follow the same formula, the total number of modified spanning trees over all values of $(p,q) \in S_{2\ell}$ which contribute to the $h^*$-polynomial of $\Sigma(\Gamma(n+1))$ for a fixed $\ell$ is indeed
    \begin{equation}\label{eq: modified spanning trees}
        \sum_{(p,q) \in S_{2\ell}} \binom{n}{2\ell}\binom{2\ell+1}{p + q + 1}.
    \end{equation}
    Since the modified spanning trees are in bijection with the spanning trees, \eqref{eq: modified spanning trees} is also the number of spanning trees of $\Gamma(n+1)$.
    
    We now study the summand contributed by each unmodified tree $T$ counted in \eqref{eq: modified spanning trees}, which is dependent on the parameters $p$ and $q$ describing $\md(T)$.
    From the structure of $\Gamma(n+1)$ and Proposition~\ref{prop: spanning tree characterization}, we will use the fact that, 
    for each of the edges $u_iu_{i+1}$ that are part of a chorded pair of $T$, exactly half of them must be 
    oriented in the same direction.
    To see why this fact holds, suppose that $m$ $(\ge \ell+1)$ chords are in the same direction.
    From Lemma \ref{lem: feasible trees} there are $n - 2 \ell$ chordless pairs and $\epsilon(T)$ unpaired edge(s).
    Let $C$ be the cycle of length $1+(2\ell-\epsilon(T)) + 2( n - 2 \ell)+ \epsilon(T)  = 2n-2\ell+1 $ consisting of $\widetilde{e}$, chords in chorded pairs, chordless pairs, and $u_i u_{i+1}$ where either $u_i v_i$ or $u_{i+1}v_i$ is an unpaired edge.
    By Corollary \ref{cor: chordless pair orientations}, at least $m+n - 2 \ell $ edges are oriented in the same direction.
    Then 
    \[
        \frac{|C|-1}{2} -  ( m+n - 2 \ell) =\ell - m <0.
    \]
    From Proposition~\ref{prop: spanning tree characterization}, this is a contradiction.
    The contribution to the $h^*$-polynomial then depends on whether the chorded pairs in $\md(T)$ are of type $\alpha$ or type $\beta$. 
    
    We can determine a valid orientation of each tree $T$ by considering which orientations satisfy Proposition~\ref{prop: spanning tree characterization}.
    If $T$ has $p$ chorded pairs of type $\alpha$ in $L(\md(T))$, we choose $i$ of them and orient the corresponding edges in $T$ away from $u_1$.
    Similarly, we choose $j$ of them from $R(\md(T))$ and orient the corresponding edges in $T$ away from $u_1$.
    There are $2\ell-p-q$ chorded pairs of $\md(T)$ remaining, all of type $\beta$; 
    we must orient the corresponding edges in $T$ so that exactly $\ell$ of the chords $u_ru_{r+1}$ (possibly including $\widetilde{e}$) are pointing in the same direction.
    By selecting which pairs are oriented so that the chord (possibly including $\widetilde{e}$) is pointing clockwise in our planar embedding of $\Gamma(n+1)$, this must be done in $\binom{2\ell-p-q}{\ell-(q-j)-i}$ ways, since $i+(q-j)$ of the edges have already been given clockwise orientations.

    For each of these (unmodified) trees, there will be $n-2\ell$ edges oriented away from $u_1$ that come from the chordless pairs, and this is true for either possible orientation of the pairs.
    From the paragraph above, there will be $2(i+j)$ edges pointing away from $u_1$ that come from chorded pairs of type $\alpha$ in the modified trees, and $2\ell-p-q$ edges pointing away from $u_1$ that come from chorded pairs of type $\beta$ in the modified trees.
    Summing over all $i$ and $j$ yields a summand of $(2t)^{n-2\ell}f_{p,q,\ell}(t)$
    in the computation of the $h^*$-polynomial.
    Multiplying this by \eqref{eq: modified spanning trees} and summing over all $(p,q) \in S_{2\ell}$, we obtain \eqref{eqn: proof first part}.

    Now, consider \eqref{eqn: proof second part}, that is, suppose $c(\md(T)) = 2\ell+1$.
    If $\widetilde{e}$ is not in $\md(T)$, then $\md(T) = T$ and so $c(T) = 2\ell+1$ as well.
    However, by Lemma~\ref{lem: feasible trees} (ii), to be a member of the triangulating tree set, $c(T)$ must be even.
    Thus, $\widetilde{e}$ is in $\md(T)$ and $\epsilon(T) = 1$.
    We can then apply an argument analogous to the one above, using Lemma~\ref{lem: breaking up coefficient} (ii), to show that there are 
    \[
        \binom{n}{2\ell+1}\binom{2\ell+1}{p+q+1}
    \]
    oriented spanning trees $T$ with $p$ and $q$ chorded pairs of type $\alpha$ in $L(\md(T))$ and $R(\md(T))$, respectively.
    Note that the argument holds for each possibility of whether $\etilde$ is in a chorded pair of type $\alpha$ or $\beta$ in $\md(T)$.
    So, we must multiply each of these by the sum of the $h^*$-polynomial contributions from these chorded pairs, which is $1+2t+t^2 = (1+t)^2$.
    This establishes \eqref{eqn: proof second part} and completes the proof.
\end{proof}



\section{Proof of Theorem~\ref{thm: gamma vec}}\label{sec: proof2}

To prove Theorem~\ref{thm: gamma vec}, we will use a generating function argument.
We will make regular use of the following pair of well-known identities:
\begin{equation}\label{eq: eqn1}
    \sum_{n \geq 0} \binom{n+k}{n} x^n =  \frac{1}{(1-x)^{k+1}}
\end{equation}
and
\begin{equation}\label{eq: eqn3}
    \sum_{n \geq 0} \binom{2n+m}{n} x^n =  \frac{1}{\sqrt{1-4x}}
    \left(\frac{1-\sqrt{1-4x}}{2x}\right)^m.
\end{equation}
Equation~\eqref{eq: eqn1} appears in \cite[Table 335]{GrahamKnuthPatashnik} and is valid for nonnegative integers $k$, while \eqref{eq: eqn3} appears as \cite[Equation (5.72)]{GrahamKnuthPatashnik} and holds for any integer $m$.

The first lemma we will need is the following.
\begin{lem}\label{hodai1}
    For all $0 \leq k \leq \ell$ we have
    \[
        \sum_{a = 0}^\ell(-1)^{k-a}4^{\ell-a} \binom{2a}{a} \binom{\ell-k}{\ell-a} = \frac{\binom{\ell}{k} \binom{2\ell}{\ell}}{ \binom{2\ell}{2k}}.
    \]    
\end{lem}

\begin{proof}
    Using \eqref{eq: eqn1} and \eqref{eq: eqn3}, the latter in the case $m=0$, we compute the generating function for the left side of the identity:
    \[\begin{aligned}
        \sum_{k \ge 0} \sum_{\ell \ge0} \sum_{a = 0}^\ell(-1)^{k-a}4^{\ell-a}\binom{2a}{a} \binom{\ell-k}{\ell-a}x^k y^\ell
        &= \sum_{k \ge 0} x^k \sum_{a \ge 0}\sum_{\ell \ge a}(-1)^{k-a}4^{\ell-a} \binom{2a}{a} \binom{\ell-k}{\ell-a}y^\ell \\
        &= \sum_{k \ge 0} x^k \sum_{a \ge 0}(-1)^{k-a}\binom{2a}{a} \sum_{\ell \ge 0}4^{\ell}\binom{\ell+a-k}{\ell}y^{\ell+a}\\
        &=\sum_{k \ge 0} x^k \sum_{a \ge 0}(-1)^{k-a}\binom{2a}{a} \frac{y^a}{(1-4y)^{a-k+1}}\\
        &= \frac{1}{1-4y}\sum_{k \ge 0} \left(-x(1-4y)\right)^k\sum_{a \ge 0}\binom{2a}{a} \left(\frac{-y}{1-4y}\right)^a\\
        &= \frac{1}{(1-4y)(1-(-x(1-4y)))\sqrt{1-4 \frac{-y}{1-4y}}}\\
        &= \frac{1}{\sqrt{1-4y} \ (1+x-4xy)}.
    \end{aligned}\]
Applying a geometric series expansion to this rational function, we obtain
\begin{align*}
\frac{1}{\sqrt{1-4y} \ (1+x(1-4y))}
&=
\sum_{k\ge 0} (-1)^k (1-4y)^{k-\frac{1}{2}} \ x^k\\
&=
\sum_{k \ge 0} \sum_{\ell \ge 0} 
(-1)^k 
(-4)^\ell 
\prod_{j=0}^{l-1} \left(k - \frac{1}{2} - j\right)
x^k \frac{y^\ell}{\ell !}\\
&=
\sum_{k \ge 0} \sum_{\ell \ge 0} 
(-1)^k 
2^\ell 
\prod_{j=0}^{l-1} (2j-2k+1)
x^k \frac{y^\ell}{\ell !}\\
&=
\sum_{k \ge 0} \sum_{\ell \ge 0} 
\frac{(2\ell-2k)! (2k)!}{(\ell -k)! k!}
x^k \frac{y^\ell}{\ell !}\\
&=
\sum_{k \ge 0} \sum_{\ell \ge 0} 
\frac{\binom{\ell}{k} \binom{2\ell}{\ell}}{ \binom{2\ell}{2k}}
x^k  y^\ell.\qedhere
\end{align*}
\end{proof}

\begin{lem} \label{hodai2}
For all $k \geq 0$,
\[
    \binom{2k}{k}\frac{x^k}{(1-4x)^{k+\frac{3}{2}}} = \sum_{\ell \ge 0} \frac{2\ell+1}{2k+1}\binom{\ell}{k}\binom{2\ell}{\ell}x^\ell.
\]  
\end{lem}

\begin{proof}
By applying Taylor expansion and elementary algebra to the left side of the identity, we have
    \begin{align*}
\binom{2k}{k}
\frac{x^k}{(1-4x)^{k+\frac{3}{2}}}
&=
\binom{2k}{k} x^k
\sum_{m \ge 0}
4^m \prod_{j=0}^{m-1} \left(k + \frac{3}{2} + j\right)
\frac{x^m}{m!}\\
&=
\binom{2k}{k} x^k
\sum_{m \ge 0}
2^m \prod_{j=0}^{m-1} (2k+3 + 2j)
\frac{x^m}{m!}\\
&=
\binom{2k}{k} x^k
\sum_{m \ge 0}
2^m 
\frac{(2k +2m+1)!!}{(2k+1)!!}
\frac{x^m}{m!}\\
&=
x^k
\sum_{m \ge 0}
2^m 
\frac{(2k)!}{k!k!}
\frac{(2k +2m+2)!}{2^{k+m+1}(k +m+1)!}
\frac{2^{k+1} (k+1)!}{(2k+2)!}
\frac{x^m}{m!}\\
&=
x^k
\sum_{m \ge 0}
\frac{(2k +2m+2)!}{2(2k+1) (k +m+1)!k!m!} \ \ 
x^m\\
&=
\sum_{\ell \ge 0}
\frac{(2\ell+2)!}{2(2k +1) (\ell+1)!k!(\ell-k)!}
x^\ell\\
&=
\sum_{\ell \ge 0}
\frac{2\ell+1}{2k+1}
\binom{\ell}{k} \binom{2\ell}{\ell}
x^\ell. \qedhere
\end{align*}
\end{proof}

The last lemma we will need is the following.

\begin{lem} \label{hodai3}
For $\ell \geq 0$,
\[
    \sum_{p,q \ge 0}\binom{2 \ell +1}{p+q+1} x^p y^q = \frac{(x+1)^{2\ell+1} - (y+1)^{2\ell+1}}{x-y}.
\]
\end{lem}

\begin{proof}
We again use a generating function argument, beginning with the left side of the desired identity multiplied by $x-y$:
\begin{align*}
(x-y)    \sum_{p,q \ge 0}
\binom{2 \ell +1}{p+q+1} x^p y^q 
&=
\sum_{p,q \ge 0}
\binom{2 \ell +1}{p+q+1} x^{p+1} y^q
-
\sum_{p,q \ge 0}
\binom{2 \ell +1}{p+q+1} x^p y^{q+1}\\
&= 
\sum_{p \ge 1}
\sum_{q \ge 0}
\binom{2 \ell +1}{p+q} x^{p} y^q
-
\sum_{p \ge 0}
\sum_{q \ge 1}
\binom{2 \ell +1}{p+q} x^p y^{q}\\
&=
\sum_{p \ge 1}
\binom{2 \ell +1}{p} x^{p}
-
\sum_{q \ge 1}
\binom{2 \ell +1}{q} y^{q}\\
&= 
\left(1+\sum_{p \ge 1}
\binom{2 \ell +1}{p} x^{p}\right)
-\left(1+
\sum_{q \ge 1}
\binom{2 \ell +1}{q} y^{q}\right)\\
&=(x+1)^{2\ell+1} - (y+1)^{2\ell+1}. \qedhere
\end{align*}
\end{proof}

\begin{thm}\label{thm: f to gamma}
For all $\ell \geq 0$ we have
    \[\sum_{p+q \le 2 \ell} \binom{2 \ell +1}{p+q+1}t^{2\ell -p-q} \sum_{i=0}^p \sum_{j=0}^q  \binom{p}{i} \binom{q}{j}\binom{2\ell -p-q}{\ell -q-i+j} t^{2i + 2j} = \sum_{a=0}^\ell \binom{2a}{a} t^a (t+1)^{4\ell -2a}.
    \]
\end{thm}

\begin{proof}
For notational convenience, set 
\[
F_\ell = \sum_{a=0}^\ell \binom{2a}{a} t^a (t+1)^{4\ell -2a}
\]
and
\[
G_{\ell,m} = 
\sum_{p+q \le 2 \ell} \binom{2 \ell +1}{p+q+1}
t^{2\ell -p-q} \sum_{i=0}^p \sum_{j=0}^q  \binom{p}{i} \binom{q}{j}
\binom{2m -p-q}{m -q-i+j} t^{2i + 2j}.\]
We compute the generating function for $F_{\ell}$ as follows:
\begin{eqnarray*}
\sum_{\ell \ge 0}  F_\ell  \lambda^\ell
&=&\sum_{\ell \ge 0} 
\sum_{a=0}^\ell \binom{2a}{a} t^a (t+1)^{4\ell -2a}
\lambda^\ell\\
&=&
\sum_{a\ge 0} 
\binom{2a}{a} t^a 
\sum_{\ell \ge a} 
(t+1)^{4\ell -2a} 
\lambda^\ell\\
&=&
\sum_{a\ge 0} 
\binom{2a}{a} t^a 
\frac{(t+1)^{2a} \lambda^a}{
1- (t+1)^4 \lambda
}\\
&=&
\frac{1}{
1- (t+1)^4 \lambda
}\sum_{a\ge 0} 
\binom{2a}{a} t^a 
(t+1)^{2a} \lambda^a\\
&=&
\frac{1}{
(1- (t+1)^4 \lambda)\sqrt{
1-4 t (t+1)^2 \lambda
}}.
\end{eqnarray*}
Hence it is enough to show that
\[
\sum_{\ell \ge 0} G_{\ell,\ell} \lambda^\ell
=\frac{1}{
(1- (t+1)^4 \lambda)\sqrt{
1-4 t (t+1)^2 \lambda
}}.
\]
Using Lemma \ref{hodai3}, we have
\[\begin{aligned}
\sum_{m \ge 0} G_{\ell, m} \lambda^m
&=
\sum_{m \ge 0}
\sum_{p, q \ge 0}
\binom{2 \ell +1}{p+q+1}
t^{2\ell -p-q} \sum_{i=0}^p \sum_{j=0}^q  \binom{p}{i} \binom{q}{j}
\binom{2 m -p-q}{m -q-i+j} t^{2i + 2j}
\lambda^m\\
&=
\sum_{p,q \ge 0}
\sum_{i=0}^p \sum_{j=0}^q  
\binom{p}{i} \binom{q}{j}
t^{2\ell -p-q + 2i + 2j}
%
\binom{2 \ell +1}{p+q+1}
%
\sum_{m \ge 0}
\binom{2m -p-q}{m -q-i+j} 
\lambda^m
\\
&=
\sum_{p,q \ge 0}
\sum_{i=0}^p \sum_{j=0}^q  
\binom{p}{i} \binom{q}{j}
t^{2\ell -p-q + 2i + 2j}
%
\binom{2 \ell +1}{p+q+1}
\lambda^{q + i -j}
\sum_{m \ge 0}
\binom{2m -p +q + 2i -2j}{m} 
\lambda^{m}
\\
&=
\sum_{p,q \ge 0}
\sum_{i=0}^p \sum_{j=0}^q  
\binom{p}{i} \binom{q}{j}
t^{2\ell -p-q + 2i + 2j}
%
\binom{2 \ell +1}{p+q+1}
\frac{ \lambda^{q + i -j}}{\sqrt{1 - 4 \lambda}}
\ \ L^{-p +q + 2i -2j}
\\
&=
\frac{t^{2\ell}}{\sqrt{1 - 4 \lambda}}
\sum_{p,q \ge 0}
\binom{2 \ell +1}{p+q+1}
\left(
\lambda L t + \frac{1}{tL}
\right)^p
\left(
\frac{\lambda L}{t}+\frac{t}{L}
\right)^q
\\
&=
\frac{t^{2\ell}}{\sqrt{1 - 4 \lambda}}
\frac{
(\lambda L t +\frac{1}{tL}+1)^{2\ell+1}
-
(\frac{\lambda L}{t}+\frac{t}{L}+1)^{2\ell+1}
}
{
(\lambda L t + \frac{1}{tL})-(\frac{\lambda L}{t}+\frac{t}{L})
},
\end{aligned}
\]
where $L =\frac{1-\sqrt{1 - 4 \lambda}}{2 \lambda}$.
It follows that $\frac{1}{L} = \frac{1+\sqrt{1 - 4 \lambda}}{2}$, that
\[\begin{aligned}
    \lambda L t +\frac{1}{tL}+1 &=  \frac{1-\sqrt{1 - 4 \lambda}}{2} t
+\frac{1+\sqrt{1 - 4 \lambda}}{2t } +1\\
&=
\frac{1}{2t} 
\left(\sqrt{1 - 4 \lambda} (1-t^2) +(t+1)^2 \right)\\
&=
\frac{t+1}{2t} 
\left(\sqrt{1 - 4 \lambda} (1-t) +(t+1) \right),
\end{aligned}\]
and that
\[\begin{aligned}
\frac{\lambda L}{t}+\frac{t}{L}+1
&=\frac{1-\sqrt{1 - 4 \lambda}}{2t}
+\frac{1+\sqrt{1 - 4 \lambda}}{2} t +1\\
&=
\frac{1}{2t} 
\left(\sqrt{1 - 4 \lambda} (t^2-1) + (t+1)^2 \right)\\
&=
\frac{t+1}{2t} 
\left(\sqrt{1 - 4 \lambda} (t-1) +(t+1) \right).
\end{aligned}
\]
Hence,
\[
\left(\lambda L t +\frac{1}{tL}\right)-\left(\frac{\lambda L}{t}+\frac{t}{L}\right) = 
\left(\lambda L t +\frac{1}{tL}+1\right)-\left(\frac{\lambda L}{t}+\frac{t}{L}+1\right) =
\frac{1-t^2}{t}
\sqrt{1-4\lambda}.
\]
Thus
\[
\begin{aligned}
    \sum_{m \ge 0} G_{\ell, m} \lambda^m
&=
\frac{(t+1)^{2\ell +1}}{2^{2 \ell +1} (1-t^2) (1 - 4 \lambda)}
\left(
\left(
\sqrt{1 - 4 \lambda} (1-t)+(t+1)
\right)^{2\ell+1}
+
\left(
\sqrt{1 - 4 \lambda} (1-t) -(t+1)
\right)^{2\ell+1}\right)\\
&=
\frac{(t+1)^{2\ell +1}}{2^{2 \ell} (1-t^2) (1 - 4 \lambda)}
\sum_{k\ge 0}
\binom{2\ell+1}{2k}
(\sqrt{1 - 4 \lambda} (1-t))^{2\ell+1-2k}
(t+1)^{2k}
\\
&=
\frac{(t+1)^{2\ell} (1-t)^{2\ell} }{2^{2 \ell} \sqrt{1 - 4 \lambda} }
\sum_{k\ge 0}
\binom{2\ell+1}{2k}
(1-4\lambda)^{\ell-k}
\left(
\frac{t+1}{1-t}
\right)
^{2k}
\\
&=
4^{-\ell} (t+1)^{2\ell} (1-t)^{2\ell}\sum_{a \ge 0}
\binom{2a}{a} \lambda^a
\sum_{k\ge 0}
\binom{2\ell+1}{2k}
(1-4\lambda)^{\ell-k}
\left(
\frac{t+1}{1-t}
\right)
^{2k}
\\
&=
4^{-\ell} (t+1)^{2\ell} (1-t)^{2\ell}
\sum_{a \ge 0}
\binom{2a}{a} \lambda^a
\sum_{k\ge 0}
\binom{2\ell+1}{2k}
\sum_{i=0}^{\ell-k}
\binom{\ell-k}{i}
(-4\lambda)^{i}
\left(
\frac{t+1}{1-t}
\right)
^{2k}.
\end{aligned}\]
Since $G_{\ell, \ell}$ is the coefficient of $\lambda^\ell$ in $\sum_{m \ge 0} G_{\ell, m} \lambda^m$,
we have
\[
G_{\ell, \ell} = 
4^{-\ell} (t+1)^{2\ell} (1-t)^{2\ell}\sum_{a \ge 0}
\sum_{k\ge 0}
\binom{2a}{a} 
\binom{2\ell+1}{2k}
\binom{\ell-k}{\ell-a}
(-4)^{\ell-a}
\left(
\frac{t+1}{1-t}
\right)
^{2k}.
\]
Additionally, from Lemma \ref{hodai1}, 
\[\begin{aligned}
G_{\ell, \ell}&=
4^{-\ell} 
(t+1)^{2\ell} (1-t)^{2\ell}
\sum_{k=0}^\ell
(-1)^{\ell-k}
\frac{2\ell+1}{2 \ell -2k+1}
\binom{\ell}{k} \binom{2\ell}{\ell}
\left(
\frac{t+1}{1-t}
\right)
^{2k}\\
&=
4^{-\ell} 
(t+1)^{2\ell} (1-t)^{2\ell}
\sum_{k=0}^\ell
(-1)^{k}
\frac{2\ell+1}{2k+1}
\binom{\ell}{k} \binom{2\ell}{\ell}
\left(
\frac{t+1}{1-t}
\right)
^{2\ell - 2k}.
\end{aligned}\]
Moreover, from Lemma~\ref{hodai2},
\[\begin{aligned}
\sum_{\ell \ge 0} G_{\ell, \ell} \lambda^\ell
&=
\sum_{\ell \ge 0} 4^{-\ell} 
(t+1)^{2\ell} (1-t)^{2\ell}\lambda^\ell
\sum_{k=0}^\ell
(-1)^{k}
\frac{2\ell+1}{2k+1}
\binom{\ell}{k} \binom{2\ell}{\ell}
\left(
\frac{t+1}{1-t}
\right)
^{2\ell - 2k} \\
&=
\sum_{k \ge 0}
(-1)^k
\left(
\frac{t+1}{1-t}
\right)
^{- 2k}
\sum_{\ell \ge 0} 
\frac{2\ell+1}{2k-1}
\binom{\ell}{k} \binom{2\ell}{\ell}
\left(
\frac{\lambda}{4}
(t+1)^4
\right)
^{\ell} \\
&=
\sum_{k \ge 0}
(-1)^k
\left(
\frac{t+1}{1-t}
\right)
^{- 2k}
\binom{2k}{k}
\frac{\left(
\frac{\lambda}{4}
(t+1)^4
\right)^k}{
\left(1 - (t+1)^4 \lambda \right)^{k+\frac{3}{2}}}.
\\
&=
\frac{1}{
\left(1 - (t+1)^4 \lambda \right)^{\frac{3}{2}}}
\sum_{k \ge 0}
\binom{2k}{k}
\left(
-
\left(
\frac{1-t}{t+1}
\right)^2
\frac{\frac{\lambda}{4}(t+1)^4}{1 - (t+1)^4 \lambda}
\right)^k\\
&=
\frac{1}{
\left(1 - (t+1)^4 \lambda \right)^{\frac{3}{2}}}
\sum_{k \ge 0}
\binom{2k}{k}
\left(
-\frac{\lambda (t+1)^2 (1-t)^2}{4(1 - (t+1)^4 \lambda)}
\right)^k\\
&=
\frac{1}{
\left(1 - (t+1)^4 \lambda \right)^{\frac{3}{2}}
\sqrt{1 +   \frac{ \lambda (t+1)^2 (1-t)^2}{1 - (t+1)^4 \lambda}}}\\
&=
\frac{1}{
(1- (t+1)^4 \lambda)\sqrt{
1-4 t (t+1)^2 \lambda
}},
\end{aligned}\]
as desired.
\end{proof}

\begin{proof}[Proof of Theorem~\ref{thm: gamma vec}]

Let $g_{n+1}(t)$ be the claimed formula for $\gamma(\Sigma(\Gamma(n+1));t)$:
\[
    g_{n+1}(t) = \sum_{\ell=0}^{\lfloor n/2 \rfloor}(2t)^{n-2\ell-1}\left(2\binom{n}{2\ell}t+\binom{n}{2\ell+1}\right)\sum_{a=0}^\ell\binom{2a}{a}t^a.
\]
To show that this is indeed the $\gamma$-polynomial for $\Sigma(\Gamma(n+1))$, we will show that
\[
    h^*(\Sigma(\Gamma(n+1)); t) = (1+t)^{2n}g_{n+1}\left(\frac{t}{(1+t)^2}\right)
\]
as in \eqref{eq: hstar gamma equivalence}.

First, observe that by simply distributing $(1+t)^{2n}$, we obtain
\[
    (1+t)^{2n}g_{n+1}\left(\frac{t}{(1+t)^2}\right) = \sum_{\ell=0}^{\lfloor n/2 \rfloor}(2t)^{n-2\ell-1}\left(2\binom{n}{2\ell}t+(1+t)^2\binom{n}{2\ell+1}\right)\sum_{a=0}^\ell\binom{2a}{a}t^a (1+t)^{4\ell - 2a}
\]
By substituting the formula from Theorem~\ref{thm: f to gamma} into the above equation, we indeed obtain $h^*(\Sigma(\Gamma(n+1));t)$.
Therefore, $g_{n+1}(t)$ is the $\gamma$-polynomial for $\Sigma(\Gamma(n+1))$, as desired.
\end{proof}

\section{Acknowledgements}

The authors would like to thank the anonymous referees for their thorough and thoughtful feedback.
Their comments led to substantial improvements in the exposition and clarity of this article.

\bibliographystyle{plain}
\bibliography{references}
\end{document}